\documentclass[10pt, a4paper]{article}

\usepackage{color}

\definecolor{ggreen}{rgb}{0,0.5,0.5}
\definecolor{brown}{rgb}{0.8,0.4,0.5}

\definecolor{lb}{rgb}{0.5,0.4,0.9}

\usepackage{times}
\usepackage[left=2.6cm, right=2.6cm, top=2.6cm, bottom=2.6cm]{geometry}

\usepackage[latin1]{inputenc}
\usepackage[T1]{fontenc}
\usepackage[english,ngerman]{babel}

\usepackage{amsthm}
\usepackage{amsmath}
\usepackage{amssymb}

\usepackage{paralist}
\usepackage[arrow, matrix, curve]{xy}

\usepackage{booktabs, array, dcolumn}
\usepackage{tabularx}

\usepackage{url}

\usepackage{float}
\usepackage{graphicx}
\usepackage{caption}
\usepackage{subcaption}

\usepackage{dsfont}

\usepackage{verbatim}

\usepackage[linkbordercolor={0 0 1}]{hyperref}

\newtheoremstyle{j}%
{3pt}%
{3pt}%
{}%
{\parindent}%
{\bfseries}%
{.}%
{.5em}%
{}%

\theoremstyle{plain}

\newtheorem*{rem*}{Remark}
\newtheorem*{concl*}{Conclusion}
\newtheorem*{theorem*}{Theorem}
\newtheorem*{cor*}{Corollary}
\newtheorem*{algo*}{Algorithm}

\newtheorem{theorem}{Theorem}[section]
\newtheorem{lem}[theorem]{Lemma}

\newtheorem{prop}[theorem]{Proposition}
\newtheorem{deff}[theorem]{Definition}
\newtheorem{rem}[theorem]{Remark}

\theoremstyle{definition}

\newtheorem*{example*}{Example}
\newtheorem{example}[theorem]{Example}
\newtheorem*{prob*}{Problem}

\newcommand{\Hil}{\mathcal{H}}

\newcommand{\Scal}{\mathcal{S}}

\newcommand{\Rcal}{\mathcal{R}}

\newcommand{\Z}{\mathbb{Z}}
\newcommand{\N}{\mathbb{N}}

\newcommand{\R}{\mathbb{R}}

\newcommand{\C}{\mathbb{C}}

\newcommand{\fhat}{\widehat{f}}

\newcommand{\phihat}{\widehat{\phi}}
\newcommand{\phiint}{\phi^{\textint}}
\newcommand{\phiright}{\phi^{\textright}}
\newcommand{\phileft}{\phi^{\textleft}}
\newcommand{\psiint}{\psi^{\textint}}
\newcommand{\psiright}{\psi^{\textright}}
\newcommand{\psileft}{\psi^{\textleft}}

\newcommand{\bitem}{\begin{itemize}}
\newcommand{\eitem}{\end{itemize}}
\newcommand{\beq}{\begin{equation}}
\newcommand{\eeq}{\end{equation}}
\newcommand{\beqn}{\begin{eqnarray*}}
\newcommand{\eeqn}{\end{eqnarray*}}
\newcommand{\ip}[2]{\left\langle#1,#2\right\rangle}

\newcommand{\cH}{{\cal H}}

\DeclareMathOperator{\textint}{b}
\DeclareMathOperator{\textright}{right}
\DeclareMathOperator{\textleft}{left}

\DeclareMathOperator{\diag}{diag}
\DeclareMathOperator{\spann}{span}
\DeclareMathOperator{\Id}{Id}

\DeclareMathOperator{\suppp}{supp \,}

\DeclareMathOperator{\diam}{diam}

\newcommand{\ti}{\textit}

\numberwithin{equation}{section}

\newcommand{\operp}{\mathop{\bigcirc\kern-12.75pt\perp}\nolimits}

\allowdisplaybreaks

\begin{document}

\selectlanguage{english}

\title{Linear Stable Sampling Rate: Optimality of 2D Wavelet Reconstructions from Fourier Measurements}

\author{Ben Adcock, Anders C. Hansen, Gitta Kutyniok, and Jackie Ma}

\maketitle

\begin{abstract}
In this paper we analyze two-dimensional wavelet reconstructions from Fourier samples within the framework of generalized
sampling. For this, we consider both separable compactly-supported wavelets and boundary wavelets. We prove that the number
of samples that must be acquired to ensure a stable and accurate reconstruction scales linearly with the number of reconstructing wavelet functions. We also provide numerical experiments that corroborate our theoretical results.
\end{abstract}


\section{Introduction}

A problem that appears in multiple disciplines is the reconstruction of an object from linear measurements. One
special situation of particular importance which we will focus on in this
paper are  Fourier measurements.
This particular reconstruction problem occurs in numerous applications
such as Fourier optics, radar imaging, magnetic resonance imaging (MRI) and X-ray CT (the latter after application of the Fourier slice theorem). One of the main issues is that we are only able
to acquire finitely many samples, since we cannot process an infinite amount of information in practice. The reconstruction
of an object from a finite collection of Fourier samples can be obtained by a truncated Fourier series. However, this typically
leads to undesirable effects such as the Gibbs phenomenon, which are wild oscillation near points of discontinuity.

Beside the Gibbs phenomenon, the convergence of the Fourier series (in the Euclidean norm) is slow.  Conversely, wavelets
bases are well-known to achieve much better results (see \cite[Ch.\ 9]{Mallat}).  Indeed, wavelets have much better localization properties than the
standard Fourier transform, leading to a better detection of image features.
For this reason, wavelets have found widespread used in compression and denoising. For example the algorithm of JPEG 2000 is based on wavelets.
Moreover, they come equipped with fast algorithms which are of great importance in today's age of technology. Wavelets also play a pivotal role in biomedical imaging, with an example being the technique of wavelet encoding in MRI (see \cite{HealyWeaver, Panych, UnserAldroubiWaveletReview, UnserAldroubiLaineEditorial}).

The issue, however, is that physical samplers such as an MRI scanner naturally yield Fourier measurements, not wavelet coefficients.  Thus in order to exploit the power of wavelets, we need a reconstruction algorithm capable of producing wavelet coefficients given a fixed set of Fourier measurements.


\subsection{Generalized sampling}

Generalized sampling is a framework for this problem developed by two of the authors in a series of papers \cite{AH2,AH1,AH3,AHHT1} and based on past work of Unser and Aldroubi \cite{unser2000sampling,UnserAld}, Eldar \cite{Eldar1} and Hrycak and Gr\"{o}chenig  \cite{hrycakIPRM}. The theory allows for stable and accurate reconstructions in an arbitrary reconstruction system of choice given fixed measurements with respect to another system.

The problem of reconstructing wavelet coefficients from Fourier samples is an important example of this abstract framework.
Mathematically, the reconstruction problem can be modeled in a separable Hilbert space $\Hil$, resulting in an infinite-dimensional linear algebra problem.  Generalized sampling provides a faithful discretization
of such a problem. The stable sampling rate (see Section \ref{sec:generalized_sampling} for details) is fundamental characteristic within generalized sampling that determines how
many samples are needed in order to obtain stable and accurate reconstructions with a given number of reconstruction elements.  It is therefore vital that this rate be determined for important instances of generalized sampling.


\subsection{Our contribution}

In \cite{AHP1} the respective authors proved the linearity of the stable sampling rate for one-dimensional compactly
supported wavelets based on finitely many Fourier samples. This means, up to a constant, one needs the same number of
samples as reconstruction elements.  Our results extend the previous one to dimension two, although higher dimensional
results can be obtained in a straightforward manner. This is an important extension, since most of the above applications involve two- or three-dimensional images.  The crucial part that makes our result non-trivial is the allowance
of non-diagonal scaling matrices neglecting straightforward arguments for separable two dimensional wavelets from 1D to
2D. Moreover, we will not only prove the linearity for standard two- dimensional separable wavelets, but also for two-dimensional boundary wavelets which are of particular interest for smooth images.  This case was not considered in \cite{AHP1} but was addressed recently in \cite{BAMGACHNonuniform1D} for the case on 1D nonuniform Fourier samples.  Here, for simplicity, we consider only uniform samples but in the 2D setting.

At this stage we note that other higher dimensional concepts, such as curvelets and shearlets, can provide better approximations
rates for cartoon-like-images; a specific class of functions (\cite{CandesDonoho}, \cite{KutLim}).  However, this paper serves as an extension of known 1D
results \cite{AHP1}. It thus provides a necessary first step in the study of reconstructions from Fourier samples within the context of
generalized sampling in higher-dimensional settings. We shall discuss shearlets in an upcoming paper.

Let us now make one further remark.  The reader may at this stage wonder why, given a vector $y$ of Fourier samples of a 2D image, one cannot simply form the vector $x = U^{-1}_{\mathrm{dft}}y$, and then form $z = V_{\mathrm{dwt}}x$ and hope that $z$ would represent wavelet coefficients of the function $f$ to be reconstructed (here $U_{\mathrm{dft}}$ and $V_{\mathrm{dwt}}$ denote the discrete Fourier and wavelet transforms respectively).  Unfortunately, $x$ represents a discretization of the truncated Fourier series of $f$. Thus, ignoring the wavelet crime \cite[p.\ 232]{Strang} for a moment, we find that $z$ represents the wavelet coefficients of the truncated Fourier series and not the wavelet coefficients of $f$ itself (taking the wavelet crime into account, $z$ would actually be an approximation to the wavelet coefficients of the truncated Fourier series). Thus, $z$ will typically have lost all the decay properties of the original wavelet coefficients. Moreover, if we map $z$ back to the image domain we get $x = V_{\mathrm{dwt}}^{-1}z$ and thus we do not gain anything as $x$ is the discretized truncated Fourier series.

This paper is about getting the actual wavelet coefficients of $f$ from the Fourier samples, thus preserving all the decay properties of the original coefficients. This is done using generalized sampling.  We show that the number
of samples that must be acquired to ensure a stable and accurate reconstruction scales linearly with the number of reconstructing wavelet functions. This means that one can reconstruct a function from its Fourier coefficients yet get error bounds on the reconstruction (up to a constant) in terms of the decay properties of the wavelet coefficients. Put in short: seeing Fourier coefficients of a function is asymptotically as good as seeing the wavelet coefficients directly.  As we will see in the numerical experiments, the generalized sampling reconstruction provides a substantial gain over the classical Fourier reconstruction.


The outline of the remainder of this paper is as follows.  In Section 2 we give a more elaborate introduction into generalized sampling and present the main results of this method.  Furthermore, we introduce the stable sampling rate to which our main focus is devoted. Having presented the framework we are dealing with, we introduce the wavelet reconstruction systems and the Fourier sampling systems in Section 3. The main results are then presented in Section 4. Finally we demonstrate our theoretical results in application by presenting some numerical experiments in Section 5.  Proofs of the main results are given in Section 6.


\section{Generalized Sampling}
\label{sec:generalized_sampling}

In this section, we recall the main definitions and results of the methodology of generalized sampling from \cite{AH2,AH1,AH3,AHHT1}.
For this, we start by introducing a general model situation for reconstruction from samples with associated quality measures.


\subsection{General Setting}

Let $\cH$ be a (separable) Hilbert space $\Hil$ with an inner product $\langle \cdot , \cdot \rangle$, which will be
our ambient space throughout this section. For modeling the acquisition of samples, let $ \Scal \subset \Hil$ be a
closed subspace and $\{s_k\}_{k \in \N} \subset \Hil$ be an orthonormal basis for $\Scal$. We will refer to
$\{s_k\}_{k \in \N}$ as the \ti{sampling system} and $\Scal$ as the \ti{sampling space}. For a signal $ f\in \Hil$,
we then assume that the associated samples (also called {\it measurements}) are given by
\beq \label{eq:samples}
    m(f)_k := \langle f, s_k \rangle, \quad k \in \N.
\eeq
Based on the measurements $m(f)=(\langle f, s_k\rangle)_{k \in \N}$, we aim to reconstruct the original signal $f$.
To be able to utilize some prior knowledge concerning the initial signal $f$, we also require a carefully chosen
\ti{reconstruction system} $\{r_i\}_{i \in \N} \subset \Hil$. The space $\Rcal = \overline{\spann}\{ r_i \, : \, i \in \N\}$,
in which $f$ is assumed to lie or be well approximated in, is then referred to as the corresponding \ti{reconstruction space}.

Since in an algorithmic realization, only finitely many samples -- and likewise a finite linear combination of
reconstruction elements -- is possible, we also introduce the finite-dimensional spaces
\[
    \Scal_M = \spann \{s_1, \ldots, s_M\}, \quad M \in \N
\]
and
\[
    \Rcal_N = \spann \{r_1, \ldots, r_N\}, \quad N \in \N.
\]
Thus, the reconstruction problem can now be phrased as follows: Given samples $m(f)_1, \ldots, m(f)_M$ of an initial
signal $f \in \cH$, compute a good approximation to $f$ in the reconstruction space $\Rcal_N$. Aiming to compare
different methodologies for solving this problem, we next formally introduce the notion of {\it reconstruction method}.

\begin{deff} \label{def:reconstructionmethod}
Let a sampling system $\{s_k\}_{k \in \N}$ and reconstruction spaces $\Rcal_N$, $N \in \N$ be defined as before.
Further, let $f \in \cH$ and let $M, N \in \N$. Then some mapping
\[
F_{N,M} : \Hil \longrightarrow \Rcal_N
\]
is called \emph{reconstruction method}, if, for every $f \in \cH$, the signal $F_{N,M}(f)$ depends only on
$m(f)_1, \ldots, m(f)_M$, where the samples $m(f)_j$ are defined in \eqref{eq:samples}.
\end{deff}

We can now strengthen the reconstruction problem in the following way: Given the dimension of the reconstruction
space $N$, how many samples $M$ are required to obtain a stable and optimally accurate reconstruction method?
This intuitive phrasing will be made precise in the next subsections.


\subsection{Quality Measures for Reconstruction Methods}
\label{subsec:QualityMeasures}

We start by introducing two quality measures for reconstruction methods that analyze the degree of approximation
within the reconstruction space and robustness for reconstruction. For this, throughout this subsection, let
$\{s_k\}_{k \in \N}$ be a sampling system, and let $\Rcal_N$, $N \in \N$ be reconstruction spaces.

The first measure quantifies the closeness of the reconstruction to the `best' reconstruction in the sense of the orthogonal
projection onto the reconstruction space. In the sequel, for the orthogonal projection onto a closed subspace $U$,
we will always utilize the notation $P_{U} : \Hil \longrightarrow U$.

\begin{deff}[\cite{AHP2}]
Let $F_{N,M}: \Hil \longrightarrow \Rcal_N$ be a reconstruction method. Moreover, let $\mu = \mu(F_{N,M})>0$ be the
least number such that
\[
    \| f - F_{N,M}(f) \| \leq \mu \| f - P_{\Rcal_N} (f) \| \quad \text{for all }  f \in \Hil.
\]
Then we call $\mu$ the \emph{quasi-optimality constant} of $F_{N,M}$. If no such constant exists, then we write
$\mu = \infty$. If $\mu$ is small, we say $F_{N,M}$ is \emph{quasi-optimal}.
\end{deff}

The second measure quantifies stability in the sense of robustness against perturbations.

\begin{deff}[\cite{AHP2}]
Let $F_{N,M}: \Hil \longrightarrow \Rcal_N$ be a reconstruction method. The \emph{(absolute) condition number}
$\kappa = \kappa(F_{N,M})>0$ is then given by
\[
    \kappa = \sup \limits_{f \in \Hil} \lim \limits_{\varepsilon \searrow 0 } \sup \limits_{\substack{ g \in \Hil, \\
    0 < \| m(g) \|_{\ell^2} \leq \varepsilon}} \left( \frac{\| F_{N,M}(f+g) - F_{N,M}(f) \|}{\|m(g)\|_{\ell^2}} \right),
\]
where $m(g)= (m(g)_1, \ldots, m(g)_M, 0, \ldots)$. If $\kappa$ is small, then we say $F_{N,M}$ is \emph{well conditioned},
otherwise $F_{N,M}$ is called \emph{ill-conditioned}.
\end{deff}

We now merge both definitions in order to have only one single measure for a reconstruction method.

\begin{deff}[\cite{AHP2,AHP1}] \label{def:reconstructionconstant}
Let $F_{N,M}: \Hil \longrightarrow \Rcal_N$ be a reconstruction method. The \emph{reconstruction constant} of $F_{N,M}$ is defined as
\[
    C(F_{N,M}) = \max \{ \mu(F_{N,M}), \kappa(F_{N,M}) \},
\]
where $\mu(F_{N,M})$ is the quasi-optimality constant and $\kappa(F_{N,M})$ is the (absolute) condition number.
\end{deff}


\subsection{The Infimum Cosine Angle}

Intuitively, the angle between the sampling and reconstruction space should play a role in determining the
reconstruction constant for a given reconstruction method. As we will see, this will be in particular the
case for the reconstruction method of generalized sampling. To prepare those results, in this subsection,
we first introduce a particularly useful notion of angle.
The concept of  principal angles between two Euclidean subspaces is well known in the literature. However, for
arbitrary closed subspaces of a Hilbert space many different notions of an angle exist. We exemplarily mention
the \ti{Friedrichs angle} and the \ti{Dixmier angle} (cf. \cite{CorachMaest,FDeutsch}). For our analysis,
the notion of infimum cosine angle -- utilized, for instance, in \cite{Ald1,UnserAld} -- will be the most
appropriate.  It is defined as follows.

\begin{deff}
Let $U, V$ be closed subspaces of $\cH$. Then the \emph{infimum cosine angle} $\cos (\omega(U,V))$ between $U$ and
$V$ is defined by
\[
   \cos (\omega(U,V)) = \inf \limits_{\substack{u \in U, \\ \|u \| =1}} \| P_V u \|, \quad \omega(U,V) \in [0, \pi/2].
\]
\end{deff}

We remark that the infimum cosine angle is not symmetric in general. For example, if $U$ and $V$ are two non-trivial
closed subspaces of $\Hil$ and $U \neq V$ with $U \subset V$, then $\cos(\omega(U,V)) =1$, whereas $\cos(\omega(V,U))=0$.
The following general characterization of pairs of closed subspaces for which the infimum cosine angle is not symmetric was
proven in \cite{BHKL1}.

\begin{lem}[\cite{BHKL1}]
For two non-trivial closed subspaces $U,V \subset \Hil$, we have $\cos(\omega(U,V)) \neq \cos(\omega(V,U))$ if and
only if one of these quantities is zero and the other is positive.
\end{lem}

A positive infimum cosine angle between the sampling and reconstruction space will be crucial for enabling reconstruction
at all. The
next theorem provides a characterization of subspaces which admit a positive infimum cosine angle.

\begin{theorem}[\cite{Tang1}]\label{Theorem:subspacecondition}
Let $U, V$ be closed subspaces of $\Hil$. Then, we have that $\cos (\omega(U,V)) >0$ if and only if
$U \cap V^\perp = \{0\}$ and $ U + V^\perp$ is closed in $\Hil$.
\end{theorem}

This characterization gives rise to the following definition.

\begin{deff}[\cite{AHP2,Ald1}]
If $\cos(\omega(U,V^\perp)) >0$ for two closed subspaces $U, V$ of $\cH$, then we say $U$ and $V$ obey the
\emph{subspace condition}. In this case, we define the associated \emph{(oblique) projection}
\[
P_{U,V}: U \oplus V\longrightarrow \Hil
\]
with range of $P$ equal to $U$ and kernel of $P$ equal to $V$.
\end{deff}

We wish to mention that oblique projections are customarily used in sampling theory, such as, for instance,
in \cite{Ald1,Eldar1,Eldar2,UnserAld}.

\subsection{Reconstruction Method of Generalized Sampling}
\label{subsec:generalizedsamplingmethod}

We are now ready to introduce the method of generalized sampling. To this end, we will always assume that the
reconstruction space $\Rcal$ and the sampling space $\Scal^\perp$ fulfill the subspace condition. In other words, we have $\cos (\omega(\Rcal,\Scal)) >0$.
For any $M \in \N$, let $P_{\Scal_M}$ be the orthogonal projection given by
\[
P_{\Scal_M} : \Hil  \longrightarrow \Scal_M, \quad  h \mapsto \sum \limits_{k=1}^M \langle h, s_k \rangle s_k.
\]
This enables us to formally define the reconstruction method of generalized sampling.

\begin{deff}\label{generalized_sampling}
For $f \in \Hil$ and $N , M\in \N$, we define the reconstruction method of \emph{generalized sampling}
$G_{N,M} : \cH \to \Rcal_N$ by
\beq \label{GS:GSequation}
    \langle P_{\Scal_M} G_{N,M}(f), r_j \rangle = \langle P_{\Scal_M} f, r_j \rangle, \quad j = 1, \ldots, N.
\eeq
We also refer to $G_{N,M}(f)$ as the \emph{generalized sampling reconstruction} of $f$.
\end{deff}

We emphasize that this is indeed a reconstruction method in the sense of Definition \ref{def:reconstructionmethod}, since
the right-hand side of \eqref{GS:GSequation} only depends on the given samples $(\ip{f}{s_k})_{k = 1}^M$ and
not on $f$ itself. Moreover, generalized sampling is a linear reconstruction method. Algorithmically, to determine $G_{N,M}(f)$, i.e., solving \eqref{GS:GSequation}, can be
phrased as the numerical linear algebra problem of computing the coefficients $\alpha^{[N]} = (\alpha_1, \ldots, \alpha_N )\in \C^N$
as the least-squares solution of
\[
U^{[M,N]} \alpha^{[N]} = m(f)^{[M]}, \qquad \mbox{where } \;
    U^{[M,N]}= \begin{pmatrix}
        u_{11} & \ldots & u_{1N} \\
        \vdots & \ddots & \vdots   \\
        u_{M1} & \ldots & u_{MN}
    \end{pmatrix}, \; u_{ij} = \langle r_j, s_i \rangle.
\]
Note that, if $U^{[M,N]}$ is well-conditioned this requires $\mathcal{O}(M N)$ operations in general.  However, in the case of Fourier samples and wavelet reconstruction one can make use of the Fast Fourier Transform and the discrete wavelet transform  and thus reduce this figure down to only $\mathcal{O}(M \log M)$ operations.

The philosophy of generalized sampling is to allow the number of samples $M$ to grow
independently of the fixed number of reconstruction elements $N$. This flexibility of $M$ and $N$ is crucial for stable reconstruction.
The next theorem guarantees the existence of the reconstruction for any $f \in \cH$ provided
that the number of samples $M$ is large enough.

\begin{theorem}[\cite{AHP2}]\label{GS:existence}
Let $N \in \N$. Then, there exists an $M_0 \in \N$, such that, for every $f \in \cH$, (\ref{GS:GSequation}) has
a unique solution $G_{N,M}(f)$ for all $M \geq M_0$. In particular, we then have
\[
G_{N,M} = P_{\Rcal_N ,(P_{\Scal_M}(\Rcal_N))^\perp}.
\]

Moreover, the smallest $M_0$ is the least number such that $\cos( \omega(\Rcal_N, \Scal_{M_0} ))>0.$
\end{theorem}

It is a priori not clear how to find $M \in \N$ large enough such that $\cos(\omega(\Rcal_N, \Scal_M)) >0$, or
even determine the smallest such value $M_0 \in \N$. In the next subsection, it will turn out that this is
intimately related to the reconstruction constant defined in Definition \ref{def:reconstructionconstant},
and will lead to the notion of a stable sampling rate.


\subsection{Stable Sampling Rate}

In Subsection \ref{subsec:QualityMeasures}, we introduced the reconstruction constant as the main quality
measure for stable and accurate reconstructions. Intriguingly, we can now relate this notion to the
infimum cosine angle in the case of generalized sampling as reconstruction method.

\begin{theorem}[\cite{AHP2}] \label{theo:generalized_sampling_Cmk}
Retaining the definitions and notations from Subsection \ref{subsec:generalizedsamplingmethod}, for all $f \in \cH$, we have
\begin{align*}
    \| G_{N,M}(f) \| \leq \frac{1}{\cos( \omega(\Rcal_N,  \Scal_M))} \| f \|,
\end{align*}
and
\begin{align*}
    \| f - P_{\Rcal_N} f \| \leq \| f- G_{N,M}(f) \| \leq  \frac{1}{\cos( \omega(\Rcal_N,  \Scal_M))}  \| f - P_{\Rcal_N} f \|.
\end{align*}
In particular, these bounds are sharp. Moreover, the reconstruction constant of generalized sampling $C(G_{N,M})$ obeys
\begin{align*}
    C(G_{N,M}) = \mu(G_{N,M}) = \kappa(G_{N,M}) = \frac{1}{\cos( \omega(\Rcal_N,\Scal_M))}.
\end{align*}
\end{theorem}

Hence, in order to obtain stable and accurate reconstructions, it is both necessary and sufficient to control the angle between $\Rcal_N$
and $\Scal_M$. This leads us to the definition of the stable sampling rate.

\begin{deff}\label{def:SSR}
For $N \in \N$ and $\theta>1$, the \emph{stable sampling rate} is defined as
\[
    \Theta(N,\theta) = \min \left\{ M \in \N \, : \, C(G_{N,M}) < \theta \right \}.
\]
\end{deff}
Note that the stable sampling rate is of interest, since it determines the number of samples required for guaranteed, quasi-optimal and numerically stable reconstructions.

We next aim to compare generalized sampling to other reconstruction methods. For this, we first introduce a class
of reconstruction methods, which recover signals from the reconstruction space exactly.

\begin{deff}
A reconstruction method $F_{N,M}: \Hil \longrightarrow \Rcal_N$ is called \emph{perfect}, if $F_{N,M} (f) = f$ whenever $f \in \Rcal_N$.
\end{deff}

Note that any reconstruction method with finite quasi-optimality constant is automatically perfect.
The next result proves that generalized sampling is in the sense superior to any other perfect reconstruction method
that its condition number as well as even its reconstruction constant is smaller.

\begin{theorem}[\cite{AHP2}]\label{theorem:optimality}
Let $M \geq N$, and let $F_{N,M}: \Hil \longrightarrow \Rcal_N$ be a linear or non linear perfect reconstruction method. Then
\[
    \kappa(G_{N,M}) \leq \kappa(F_{N,M}),
\]
where $G_{N,M}$ is the reconstruction method of generalized sampling. In particular, we have
\[
C(G_{N,M}) \leq C(F_{N,M}).
\]
\end{theorem}

Next, we aim to compare the quasi-optimality constant of generalized sampling to other reconstruction methods. For this,
assume that the stable sampling rate of generalized sampling is linear in $N$, i.e., $\Theta(N, \theta) = \mathcal{O}(N)$ as
$N \rightarrow \infty$, and assume that there exist constants $C_f,D_f, \gamma_f >0$ depending on the initial signal $f \in \cH$
such that
\beq \label{eq:decayoptthm}
    C_f N^{-\gamma_f} \leq \| f - P_{\Rcal_N}f \| \leq D_f N^{-\gamma_f}, \quad \mbox{for all } N \in \N.
\eeq
The following result then shows that the error of any reconstruction method $F_M$ can be only up to a constant
better than the error of generalized sampling reconstruction.

\begin{theorem}[\cite{AHP2}] \label{theo:general_linear_rate}
Suppose that the stable sampling rate $\Theta(N, \theta)$ is linear in $N$, i.e. $\Theta(N, \theta) = \mathcal{O}(N)$ as
$N \rightarrow \infty$. Let $f \in \Hil$ be fixed and let
\begin{align*}
    F_M: (m(f)_1, \ldots, m(f)_M) \mapsto F_M(f) \in \Rcal_{\psi_f(M)},
\end{align*}
be a reconstruction method, where $\psi_f: \N \longrightarrow \N$ with $\psi_f(M) \leq \lambda M$ for some $\lambda >0$.
Assume that \eqref{eq:decayoptthm} holds. Then, for any $\theta >1$, there exist constants $d(\theta) \in (0,1)$ and
$c(\theta, C_f, D_f)>0$ such that
\[
    \| f - G_{d(\theta)M,M}(f) \| \leq c(\theta, C_f, D_f) \| f - F_M(f) ||, \quad \mbox{for all } M \in \N,
\]
where $G_{N,M}$ denotes the generalized sampling reconstruction method.
\end{theorem}


\section{Linear Sampling Rate for Compactly Supported 2D Wavelets}

As already elaborated upon in the introduction, the situation of taking the Fourier basis as sampling system is of
particular importance to applications. A very common choice for a reconstruction system in imaging sciences are
2D wavelets, predominantly of compact support due to their high spatial localization.

In this section, we will state our main result for the situation of Fourier samples and reconstruction within a
wavelet basis using 2D compactly supported wavelets. More precisely, we will show linearity of the stable sampling
rate for the associated generalized sampling scheme, which shows by Theorem \ref{theo:general_linear_rate} the
near-optimality of this reconstruction method.
For this, we start by defining first the reconstruction and second the sampling space, followed by the statement of our main
result. Due to its technical nature, we present its proof in a separate section, namely Section \ref{sec:proofWavelets}.


\subsection{Sampling and Reconstruction Spaces}


\subsubsection{Compactly Supported 2D Wavelets}
We start by recalling the notion of scaling matrices.
\begin{deff}
Let $A$ be a $2 \times 2$ matrix with non-negative integer entries and eigenvalues greater than one in modulus. Then we
call $A$ a \emph{scaling matrix}.
\end{deff}
For the sake of brevity, in the sequel, we will use the notation
\[
    A = \begin{pmatrix} \lambda_1 & \lambda_2 \\ \lambda_3 & \lambda_4 \end{pmatrix}.
\]
Moreover, for the entries of $A^j$ we write
\[
    A^j = \begin{pmatrix} \lambda_1^{(j)} & \lambda_2^{(j)} \\ \lambda_3^{(j)} & \lambda_4^{(j)} \end{pmatrix}, \quad j \in \N.
\]
Notice that $\lambda_i^{(j)} \neq (\lambda_i)^j$ in general. For the sake of completeness, we next give the definition of a 2D multiresolution analysis (MRA). For more details, we refer to the existing literature, e.g \cite{Dau,Maass}.

\begin{deff}
Let $A$ be a scaling matrix. Then a sequence of closed subspaces $(V_j)_{j \in \Z}$ of $L^2(\R^2)$ is called a
\emph{multiresolution analysis}, if the following properties are satisfied.
\begin{compactenum}[i)]
\item $ \{0\} \subset \ldots \subset V_j \subset V_{j+1} \subset \ldots \subset L^2(\R^2)$,
\item $\bigcap \limits_{j \in \Z} V_j = \{0 \}$,
\item $\overline{\bigcup \limits_{j \in \Z} V_j} = L^2(\R^2)$,
\item $f \in V_j \Leftrightarrow f(A \cdot) \in V_{j+1}$,
\item there exists a function $\phi \in L^2(\R^2)$ (called \emph{scaling function}), such that
\[
\{  \phi_{0,m} := \phi ( \cdot - m ) \, : \, m \in \Z^2 \}
\]
constitutes an orthonormal basis for $V_0$.
\end{compactenum}
\end{deff}

The associated {\it wavelet spaces} $(W_j)_{j \in \Z}$ are then defined by
\[
 V_{j+1} = V_j \oplus W_j.
\]
It is well known that there exist $| \det A | -1$ corresponding compactly supported
wavelets $\psi^1, \ldots, \psi^{|\det A| -1}$ such that
\[
\{\psi_{j,m}^p := | \det A |^{j/2} \psi^p(A^j \cdot -m) :  m= (m_1,m_2) \in \Z^2, p=1, \ldots, |\det A| -1\}
\]
forms an orthonormal basis for $W_j$ for each $j$, see, e.g., \cite{Meyerondelettes}. We now consider the decomposition
\[
L^2(\R^2) = V_0 \oplus \bigoplus_{j \in \N} W_j,
\]
where
\[
 V_0 :=\overline{\spann}\{ \phi_{0,m} \, : \, m= (m_1,m_2) \in \Z^2\}
\]
and
\[
W_j := \overline{\spann}\{ \psi^p_{j,m} \, : \, m= (m_1,m_2) \in \Z^2, p = 1,  \ldots, | \det A | -1 \}, \quad j \in \N.
\]


\subsubsection{Reconstruction Space}
\label{subsec:reconstructionspace}

Our aim is to reconstruct functions that are supported on $[0,a]^2$. To this end, suppose that the scaling function and wavelet functions are supported in $[0,a]^2$.
To mimic the fact that practical applications can only handle finite systems, we
restrict to those functions whose support intersects $[0,a]^2$, i.e., to the systems
\[
    \Omega_1=\{ \phi_{0,m} \, : \, m= (m_1,m_2) \in \Z^2, -a< m_1, m_2 < a \}
\]
and
\begin{align*}
    \Omega_2 = \{ \psi^p_{j,m} \, : \, j \in \N \cup \{ 0 \}, \ &m= (m_1,m_2) \in \Z^2, -a < m_1 < a (\lambda_1^{(j)}
    + \lambda_2^{(j)}), \\  & -a< m_2 < a (\lambda_3^{(j)} + \lambda_4^{(j)}),   p = 1,  \ldots, | \det A | -1  \}.
\end{align*}
The reconstruction space $\Rcal$ is then defined as the closed linear span of these functions, which is
\[
    \Rcal = \overline{\spann} \{ \varphi \, : \, \varphi \in \Omega_1 \cup \Omega_2\}.
\]
To define the finite-dimensional subspaces $\Rcal_N$, we require an ordering for this system. The most natural way to
order $\Omega_1 \cup \Omega_2$ is starting from wavelets at coarsest scale and then continue to higher scales. Within
one scale, one might order the translation $(m_1,m_2)$ in a lexicographical manner starting from the smallest number
up the largest. More precisely, we fix $m_1$ and let $m_2$ run, increase $m_1$ by one and repeat. This then leads to
the following ordering of $\Omega_1 \cup \Omega_2$:
\begin{eqnarray}\nonumber
\lefteqn{\{\varphi_i\}_{i \in \N}}\\
& = & \{ \phi_{0, (-a+1,-a+1)}, \ldots \phi_{0, (-a+1,a-1)},  \phi_{0, (-a + 2,-a+1)} \ldots, \phi_{0, (-a+2,a-1)}, \ldots,
\phi_{0, (a-1,a-1)}, \nonumber \\
&& \psi^1_{0, (-a +1 , -a +1)}, \ldots, \psi^1_{0, (-a +1 ,  m_2^{(0)} -1)},  \ldots , \psi^1_{0, (m_1^{(0)} -1 , -a +1)},
\ldots, \psi^1_{0, (m_1^{(0)} -1 ,  m_2^{(0)} -1)}, \ldots, \nonumber \\
&& \psi^{| \det A | -1}_{0, (-a +1 , - a +1)}, \ldots, \psi^{| \det A | -1}_{0, (-a +1 ,  m_2^{(0)} -1)},   \ldots ,
\psi^{| \det A | -1}_{0, (m_1^{(0)} -1 , - m_2^{(0)} +1)}, \ldots, \psi^{| \det A | -1}_{0, (m_1^{(0)} -1 ,  m_2^{(0)} -1)},  \nonumber \\
&& \psi^1_{1, (-a +1 , - m_2^{(1)} +1)}, \ldots, \psi^1_{1, (-a +1 ,  m_2^{(1)} -1)}, \psi^1_{1, (-a +2 , - m_2^{(1)} +1)},
\ldots, \psi^1_{1, (-a +2 ,  m_2^{(1)} -1)}, \ldots \} \label{eq:ordering}
\end{eqnarray}
where $m_1^{(j)} =  a (\lambda_1^{(j)} + \lambda_2^{(j)})$ and $m_2^{(j)} =  a (\lambda_3^{(j)} + \lambda_4^{(j)}) $.
We emphasize that the presented results do not depend on the specific ordering within the scale. Finally, based on
the chosen ordering, we define the reconstruction space $\Rcal_N$ by
\[
    \Rcal_N = \spann \{ \varphi_i \, : \, i = 1, \ldots, N\}, \quad N \in \N.
\]

In practise, it is much more common to use as approximation spaces those being generated up to a specific scale, say, $J$.
To mimic this approach, we coarsen the choices of $N \in \N$ for which we consider $\Rcal_N$ suitably in the following way.
First, we observe that the number of functions being in $\Rcal$ up to a fixed scale $J-1 \in \N$ is
\beq
\label{eq:N_J}
N_J = (2a - 1 )^2 + (|\det A|-1) \sum _{j=0}^{J-1} ( a(\lambda_1^{(j)} + \lambda_2^{(j)} +1) -1)( a(\lambda_3^{(j)} + \lambda_4^{(j)} +1) -1).
\eeq

\begin{lem}
Retaining the definitions and notations above, then $N_J \le C_a (\lambda_1^{(J)} + \lambda_2^{(J)})(\lambda_3^{(J)} + \lambda_4^{(J)})$ for some constant $C_a$ depending on $a$.
\end{lem}

\begin{proof}
The total number of elements $N_J$ in the reconstruction space $\Rcal_{N_J}$ up to a scale $J-1$ is prescribed by the matrix $A$ and
the support of the scaling function $\phi$ and the wavelet $\psi$, respectively. However, since the support of both $\phi$ and
$\psi$ is a square of the form $[0,a]^2$, the total number of elements in the reconstruction space $N_J$ is the same as if we would
have generated the wavelet system with the transpose of $A$. Hence, it is sufficient to prove
\begin{align*}
N_J \lesssim (\lambda_1^{(J)} + \lambda_2^{(J)})(\lambda_3^{(J)} + \lambda_4^{(J)}),
\end{align*}
where the constant depends on $a$. We prove the result by induction. For $J = 1$ it is clearly true. Hence,
\begin{align*}
(2a - 1 )^2 + (|\det A|-1) \sum _{j=0}^{J-1}& ( a(\lambda_1^{(j)} + \lambda_2^{(j)} +1) -1)( a(\lambda_3^{(j)} + \lambda_4^{(j)} +1) -1) \\
& \lesssim (\lambda_1^{(J-1)} + \lambda_2^{(J-1)})(\lambda_3^{(J-1)} + \lambda_4^{(J-1)}) + (|\det A|-1)( a(\lambda_1^{(J-1)} \\
& \quad + \lambda_2^{(J-1)} +1) -1)( a(\lambda_3^{(J-1)} + \lambda_4^{(J-1)} +1) -1)
\end{align*}
Now,
\begin{align*}
A^J &= \begin{pmatrix} \lambda_1^{(J)} & \lambda_2^{(J)} \\ \lambda_3^{(J)} & \lambda_4^{(J)} \end{pmatrix}
= \begin{pmatrix} \lambda_1^{(J-1)} & \lambda_2^{(J-1)} \\ \lambda_3^{(J-1)} & \lambda_4^{(J-1)} \end{pmatrix}\begin{pmatrix} \lambda_1 & \lambda_2 \\ \lambda_3 & \lambda_4 \end{pmatrix}
= \begin{pmatrix} \lambda_1^{(J-1)}\lambda_1 + \lambda_2^{(J-1)}\lambda_3 & \lambda_1^{(J-1)}\lambda_2 + \lambda_2^{(J-1)}\lambda_4 \\ \lambda_3^{(J-1)}\lambda_1 + \lambda_4^{(J-1)}\lambda_3 & \lambda_3^{(J-1)}\lambda_2 + \lambda_4^{(J-1)}\lambda_4 \end{pmatrix}
\end{align*}
which implies
\begin{align*}
\lambda_1^{(J)} + \lambda_2^{(J)} &= \lambda_1^{(J-1)}( \lambda_1 + \lambda_2) + \lambda_2^{(J-1)}(\lambda_3 + \lambda_4),\\
\lambda_3^{(J)} + \lambda_4^{(J)} &= \lambda_3^{(J-1)}( \lambda_1 + \lambda_2) + \lambda_4^{(J-1)}(\lambda_3 + \lambda_4).
\end{align*}
\end{proof}

Similar to the construction in \cite{AHP1}, we now define the truncated scaling space by
\[
    V^{(a)}_0 := \spann \{ \phi_{0,m} \, : \, m= (m_1,m_2) \in \Z^2, -a< m_1, m_2 < a \}
\]
and the truncated wavelet spaces by
\begin{align*}
    W^{(a)}_j := \spann \{ \psi^p_{j,m} \, : \, \  &m= (m_1,m_2) \in \Z^2, -a < m_1 < a(\lambda_1^{(j)}+ \lambda_{2}^{(j)}),  \\
    &-a < m_2 < a(\lambda_3^{(j)}+ \lambda_{4}^{(j)}),    p = 1,  \ldots, |\det A| -1  \}.
\end{align*}
The reconstruction space of interest to us is then defined by
\[
\Rcal_{N_J} = V^{(a)}_0 \oplus W^{(a)}_0 \oplus \ldots \oplus W^{(a)}_{J-1}.
\]


\subsubsection{Sampling Space}
\label{subsec:samplingspace}

To define the sampling space consisting of elements of the Fourier basis, we first choose $T_1,T_2>0$ sufficiently large such that
\[
    \Rcal \subset L^2([-T_1,T_2]^2).
\]
Thus, only functions supported on $[-T_1,T_2]^2$ are relevant to us. Indeed, choosing $T_1 \geq a-1$ and $T_2 \geq 2a-1$ is enough. Note that $\lambda_1^{(0)}=\lambda_4^{(0)} =1$ and $\lambda_2^{(0)}=\lambda_3^{(0)} =0$.
To allow an arbitrarily dense sampling, for each $\varepsilon \leq \frac{1}{T_1 + T_2}$, we define the sampling vectors
by
\beq \label{s_l}
    s_l^{(\varepsilon)} = \varepsilon e^{2 \pi i \varepsilon \langle l, \cdot \rangle} \cdot \chi_{[-T_1,T_2]^2}, \quad l \in \Z^2.
\eeq
Thus, we sample in each direction with the same sampling rate $\varepsilon$. Based on these sampling vectors, we now define
the sampling space $\Scal^{(\varepsilon)}$ by
\[
    \Scal^{(\varepsilon)} = \overline{\spann} \left\{ s_l^{(\varepsilon)} \, : \, l \in \Z^2 \right\}.
\]
The finite-dimensional subspaces $\Scal^{(\varepsilon)}_{M}$, $M = (M_1,M_2 )\in \N \times \N$, are then given by
\begin{align*}
    \Scal^{(\varepsilon)}_{M} = \spann \left\{ s_l^{(\varepsilon)} \, : \, l = (l_1,l_2) \in \Z^2, - M_i \leq l_i \leq M_i , i =1,2 \right\}.
\end{align*}


\subsection{Main Result}

Our main results concerns the stable sampling rate of the generalized sampling scheme for the sampling spaces
$\Scal^{(\varepsilon)}_M$ and the reconstruction spaces $\Rcal_{N_J}$. By Theorem \ref{theo:generalized_sampling_Cmk},
for this, we have to control the infimum cosine angle between the $\Rcal_{N_J}$ and $\Scal^{(\varepsilon)}_M$.
In particular, we wish to determine $M=(M_1,M_2) \in \N \times \N$ such that, for given $\theta>1$ and $J-1 \in \N$
the term $\cos(\omega(\Rcal_{N_J},\Scal_M^{(\varepsilon)}))$ can be bounded from below by $\theta^{-1}$, where $N_J$ denotes
the total number of reconstruction elements up to scale $J-1$.

For this to work, we will assume that the scaling matrix does not distort the grid too much. To make this
precise, let $I_M = \{Ê(l_1,l_2) \in \Z^2 \, : \, -M_i \leq l_i \leq M_i, i =1,2 \}, L_M = [-M_1,M_1]\times[-M_2,M_2]$ for $(M_1,M_2) \in \N^2$. Then, we assume that the so-called \emph{mesh norm} $\delta$ obeys
\begin{align}
 \delta:=    \max \limits_{x \in \varepsilon A^{-J}(L_M)} \min \limits_{k \in \varepsilon A^{-J}(\Z^2)} \| x - y + k\|_\infty  < \frac{\log\left(\frac{1}{\sqrt{\mu (L_M)}}+1\right)}{ 4 \pi \max\{|L_1|,|L_2|,|L_3|,|L_4|\}}, \label{assumption}
\end{align}
for some $\varepsilon$ independent of $J$, where $\mu$ denotes the 2D lebesgue measure and $L_i, i =1,2,3,4$ are bounds that are obtained by Lemma \ref{lemma:V_0plusW_j}. We will also discuss this assumption in Example \ref{wavehaarexp} for better understanding.

The following result shows that the stable sampling rate is indeed linear in the considered situation, showing
that this scheme is superior to any other reconstruction method in the sense of Theorem \ref{theo:general_linear_rate}.

\begin{theorem}\label{maintheoremgeneral}
Let $N_J$ be the number of reconstruction elements up to a fixed scale $J-1$, and let $\Rcal_{N_J}$ and $\Scal_M^{(\varepsilon)}$
be the reconstruction space and sampling space, respectively. Furthermore, assume (\ref{assumption}) is fulfilled. If $\theta>1$, then there exists a constant
$S^{(\theta)}$ independent of $J$ and $\varepsilon$ such that, if
\[
M_1 \geq \frac{ \lambda_1^{(J)} +  \lambda_3^{(J)}}{\varepsilon} S^{(\theta)},
\qquad
M_2 \geq \frac{ \lambda_2^{(J)} +  \lambda_4^{(J)}}{\varepsilon} S^{(\theta)},
\]
then, for $M = (M_1,M_2)$,
\[
\cos(\omega(\Rcal_{N_J},\Scal_M^{(\varepsilon)})) \geq \frac{1}{\theta}.
\]
In particular, the stable sampling rate obeys $\Theta(N_J,\theta) = \mathcal{O}(N_J)$ as $N_J \rightarrow \infty$ for every
fixed $\theta >1$.
\end{theorem}

Summarizing, this result shows that the required number of Fourier samples is optimally small up to a constant when utilizing 2D compactly
supported wavelets for the generalized sampling reconstruction.  This extends the result of \cite{AHP1} to the two-dimensional case.
We next consider a special, yet widely used choice for scaling matrices, namely diagonal matrices. This includes
two dimensional dyadic wavelets, which are those typically used in applications. This situation is also considered
in the numerical experiments presented in Section \ref{sec:numerics}.

\begin{example}\label{wavehaarexp}
In this example we want to demonstrate our results and, in particular, give some better understanding of the construction and assumption (\ref{assumption}). Furthermore, this example shall show that there are large classes of 2D wavelets that fulfill (\ref{assumption}). For this purpose, let
\[
    A = \begin{pmatrix} 2 & 0 \\ 0 & 2 \end{pmatrix}
\]
be the scaling matrix, which -- as mentioned before -- gives rise to three wavelet generators. Let $\phi$ and
$\psi^p, p = 1,2,3$ be two dimensional scaling and wavelet functions with compact support in $[0,a]^2, a\in\N$,
which might be obtained by tensor products of one dimensional scaling and wavelet functions, see \cite{Dau, Mallat}.

As in Subsection \ref{subsec:reconstructionspace}, we define
\begin{equation}\label{Omega1}
    \Omega_1:=\{ \phi_{0,m} \, : \, m= (m_1,m_2) \in \Z^2, |m_i | < a, i =1,2 \}
\end{equation}
and
\begin{equation}\label{Omega2}
   \Omega_2:= \{ \psi^p_{j,m} \, : \, j \in \N \cup \{0\}, m= (m_1,m_2) \in \Z^2, -a < m_i < 2^j a, i =1,2 , p = 1,2,3\},
\end{equation}
since these are the only functions whose support intersects $[0,a]^2$. Again, in line with our previous approach, we
then define the reconstruction space $\Rcal$ by
\[
    \Rcal = \overline{\spann} \{ \varphi \, : \, \varphi \in \Omega_1 \cup \Omega_2\},
\]
order the elements $\Omega_1 \cup \Omega_2$ analogously to (\ref{eq:ordering}), and set
\[
    \Rcal_N = \spann \{ \varphi_i \, : \, i = 1, \ldots, N\}, \quad N \in \N.
\]
Now, each function $\varphi$ in $\Rcal_N$ can be represented as a linear combination of scaling functions at highest level $J$, in particular, there exist positive integers $L_1,L_2,L_3,$ and $L_4$ such that
\begin{align*}
\varphi = \sum_{l_1=L_3}^{L_1} \sum_{l_2=L_4}^{L_2} \alpha_{l_1,l_2} \phi_{J,(l_1,l_2)}, \quad \alpha_{l_1,l_2} \in \C.
\end{align*}
This statement will be proven in Lemma \ref{lemma:V_0plusW_j}. With a view to (\ref{assumption}), the explicit expression of the bounds $L_i, i = 1,2,3,4$ are highly important. In fact, we should not choose them too large, otherwise (\ref{assumption}) might not hold. Since
\begin{align*}
\varphi = \sum_{l_1=L_3}^{L_1} \sum_{l_2=L_4}^{L_2} \langle \varphi, \phi_{J,(l_1,l_2)} \rangle \phi_{J,(l_1,l_2)},
\end{align*}
and is compactly supported, shifting by $(l_1,l_2)$ far enough leads to the fact that the coefficients become zero. For the scaling matrix $A= \diag(2,2)$ we have (see proof of Lemma \ref{lemma:V_0plusW_j})
\begin{align*}
L_1=L_3= 2^{J}(3a-1), \quad L_2 =L_4= -a+2^{J}(-a+1).
\end{align*}
Moreover, the mesh norm $\delta$ obeys
$
\delta \leq \frac{\varepsilon}{2^J}.
$
Hence, for $\varepsilon = \frac{1}{4\pi(3a-1)}$
\begin{align}
\delta \leq \frac{1}{4\pi(3a-1)2^J} \leq \frac{\log\left(\left(\frac{2^J}{\varepsilon}\right)^2+1\right)}{2\pi (3a-1)2^J}, \label{deltawaveineq}
\end{align}
so assumption (\ref{assumption}) is fulfilled. 
Thus, we can apply  Theorem \ref{maintheoremgeneral} to obtain a constant $S^\theta$ independent of $J$, such that for
\[
    M_1 , M_2 = \left \lceil \frac{2^J S^\theta}{\varepsilon}\right \rceil
\]
we have
\[
    \cos(\omega(\Rcal_{N_J},\Scal_M^{(\varepsilon)}))> \frac{1}{\theta}.
\]
Counting the numbers of elements in the space $\Rcal_{N_J}$ gives
\beq
N_J =(2a - 1)^2 + 3\sum \limits_{j=0}^{J-1} (2^ja + a-1)^2 = (2^{2J}-1)a^2 + 6a(a-1)(2^J-1) + 3J(a-1)^2 + (2a-1)^2 = \mathcal{O}(2^{2J}),  \label{levelnumber}
\eeq
which means the number of samples scales linearly with the number of reconstruction elements.
\end{example}


\section{Linear Sampling Rate for 2D Boundary Wavelets}

In applications, typically signals on a bounded domain are considered such as on the interval $[0,1]$. Aiming to avoid
artifacts at the boundaries $x=0, 1$, in \cite{CohDauVial} (see also \cite{Mallat}) boundary wavelets were introduced. The
considered wavelet system then consists of interior wavelets, which do not touch the boundary, and the just mentioned
boundary wavelets, leading to a system with an associated multiresolution analysis and prescribed vanishing moments
even at the boundary. Similar to the classical wavelet construction, this system can be lifted to 2D by tensor products.

In this section, we will show that in fact also this (boundary) wavelet system allows a linear stable sampling rate,
though with a proof differing significantly from the one of the `classical wavelet case' due to the different structure
at the highest scale. For this, we start with formally introducing 2D boundary wavelets -- which we use as an expression
for the whole wavelet system -- followed by our choice of reconstruction and sampling space.


\subsection{Construction of Boundary Wavelets}
\label{subsec:constructionboundary}

We start with the 1D construction of wavelets on the interval as introduced in \cite{CohDauVial}. For this, let $\phi$
be a compactly supported Daubechies scaling function with $p$ vanishing moments. It is well known that $\phi$ must then
have a support of size $2p-1$. By a shifting argument, we can assume that $\suppp(\phi) = [-p+1,p]$. In order to properly
define \ti{interior wavelets} and \ti{boundary wavelets}, the scale need to be large enough -- i.e., the support of the
wavelets need to be small enough -- to be able to distinguish between wavelets whose support fully lie in $[0,1]$ and
wavelets whose support intersect the boundary $x=0$ and, likewise, $x =1$.

Therefore, we now let $j \in \N$ such that $2p \leq 2^j$. Then there exist $2^j - 2p$ \ti{interior scaling functions}, i.e.,
scaling functions which have support inside $[0,1]$, defined by
\[
    \phiint_{j,n} = \phi_{j,n} = 2^{j/2} \phi(2^j\cdot - n), \quad \text{for } p\leq n < 2^j-p.
\]
Depending on boundary scaling functions $ \{\phileft_n\}_{n=0, \ldots, p-1}$ and $\{\phiright_n\}_{n=0, \ldots, p-1}$,
which we will introduce below, the $p$ \ti{left boundary scaling functions} are defined by
\[
    \phiint_{j,n} =  2^{j/2} \phileft_n(2^j\cdot), \quad \text{for } 0 \leq n < p,
\]
and the $p$ \ti{right boundary scaling functions} are
\[
    \phiint_{j,n} =  2^{j/2} \phiright_{2^j-1-n}(2^j(\cdot-1)), \quad \text{for } 2^j-p \leq n < 2^j.
\]
We remark that this leads to $2^j$ scaling functions in total, which is the number of original scaling functions $(\phi_{j,n})_n$
that intersect $[0,1]$.

We next sketch the idea of the construction of boundary scaling functions $ \{\phileft_n\}_{n=0, \ldots, p-1}$ as well as
$\{\phiright_n\}_{n=0, \ldots, p-1}$ following \cite{CohDauVial}, to the extent to which we require it in our
proofs. One starts by defining edge functions $\widetilde{\phi}^k$ on the positive axis $[0,\infty)$ by
\begin{align*}
    \widetilde{\phi}^k(x) = \sum_{n=0}^{2p-2} \binom{n}{k} \phi(x+n-p+1), \quad k = 0, \ldots, p-1,
\end{align*}
such that these edge functions are orthogonal to the interior scaling functions and such that they together generate all
polynomials up to degree $p-1$. After performing a Gram-Schmidt procedure one obtains the left boundary functions
$\phileft_k, k =0, \ldots, p-1$. The right boundary functions are then -- after some minor adjustments -- obtained by reflecting
the left boundary functions. This construction from \cite{CohDauVial} allows one to obtain a multiresolution analysis.

\begin{theorem}[\cite{CohDauVial}]
If $2^j \geq 2p$, then $\{\phiint_{j,n}\}_{n=0, \ldots, 2^j-1}$ is an orthonormal basis for a space $V^{\textint}_j$ that is
nested, i.e.
\begin{align*}
    V_j^{\textint} \subset V_{j+1}^{\textint},
\end{align*}
and complete, i.e.
\begin{align*}
    \overline{\bigcup_{j \geq \log_2{2p}} V^{\textint}_j} = L^2[0,1].
\end{align*}
\end{theorem}

Next, we define an orthonormal basis for the wavelet space $W^{\textint}_j$, which is as usual defined as the orthogonal
complement of $V_j^{\textint}$ in $V_{j+1}^{\textint}$. For this, let $\psi$ be the corresponding wavelet function to
$\phi$ with $p$ vanishing moments and $\suppp \psi = [-p+1,p]$. Similar to the construction of the scaling functions,
we will obtain interior wavelets and boundary wavelets, which then constitutes the set of wavelets in the interval.
Again based on a careful choice of boundary wavelets $(\psileft_n)_n$ and $(\psiright_k)_k$, for which we refer to
\cite{CohDauVial}, we define $2^j-2p$ \ti{interior wavelets} by
\[
    \psiint_{j,n} = \psi_{j,n} = 2^{j/2} \psi(2^j\cdot - n), \quad \text{for } p\leq n < 2^j-p,
\]
$p$ \ti{left boundary wavelets}
\[
    \psiint_{j,n} =  2^{j/2} \psileft_n(2^j\cdot), \quad \text{for } 0 \leq n < p,
\]
and $p$ \ti{right boundary wavelets}
\[
    \psiint_{j,n} =  2^{j/2} \psiright_{2^j-1-n}(2^j(\cdot-1)), \quad \text{for } 2^j-p \leq n < 1^j.
\]

Summarizing, the following result hold for these wavelet functions.

\begin{theorem}[\cite{CohDauVial}]\label{theorem:1Dboundarywaveletsoverview}
Let $2^J \geq 2p$. Then the following properties hold:
\begin{compactenum}[i)]
\item $\{ \psiint_{J,n} \}_{n = 0, \ldots, 2^J-1}$ is an orthonormal basis for $W_J^{\textint}$.
\item $L^2[0,1]$ can be decomposed as
\[
L^2[0,1]  = V_J^{\textint} \oplus W_{J}^{\textint} \oplus W_{J+1}^{\textint} \oplus W_{J+2}^{\textint} \oplus \ldots =
V_J^{\textint} \bigoplus \limits_{j=J}^{\infty} W_{j}^{\textint}.
\]
\item $\left\{  \{ \phiint_{J,m} \}_{m = 0, \ldots, 2^J-1}, \{ \psiint_{j,n} \}_{j \geq J,n = 0, \ldots, 2^j-1}\right\}$
is an orthonormal basis for $L^2[0,1]$.
\item If $\phi, \psi \in C^r[0,1]$, then $\left\{  \{ \phiint_{j,m} \}_{m = 0, \ldots, 2^J-1}, \{ \psiint_{j,n} \}_{j \geq J,n = 0,
 \ldots, 2^j-1}\right\}$ is an unconditional basis for $C^s[0,1]$ for all $s <r$.
\end{compactenum}
\end{theorem}

As mentioned before, this system gives rise to a 2D system by tensor products, i.e., by the standard 2D
separable wavelet construction. In particular, this 2D system again constitutes an MRA, see \cite{Mallat}.


\subsection{Sampling and Reconstruction Space}
\label{subsec:samplingreconstruction_b_wavelets}


\subsubsection{Reconstruction Space}

For defining the reconstruction space, we will now assume that the region of interest is $[0,1]^2$ instead of
$[0,a]^2$ with $a \in \N$, as previously chosen. Our starting point is a 1D compactly supported Daubechies scaling
function with $p$ vanishing moments (cf. Subsection \ref{subsec:constructionboundary}). Recall that Daubechies wavelets have at least the following frequency decay,
\beq
    |\phihat(\xi) | \lesssim \frac{1}{1 + |\xi|}, \quad \text{ as } |\xi| \rightarrow \infty. \label{eq:freqdecay}
\eeq
Let $\psi$ be a corresponding wavelet with $p$ vanishing moments, and let $\phiint$ and $\psiint$ be the wavelets on the
interval as introduced in the previous subsection. Let $J_0$ be the smallest number such that $2^{J_0} \geq 2p$. Then the
associated 2D scaling functions are of the form
\[
    \phiint_{J_0,(n_1,n_2)} := \phiint_{J_0,n_1} \otimes \phiint_{J_0,n_2}, \quad 0\leq n_1,n_2 \leq 2^{J_0}-1
\]
and, for $0 \leq n_1,n_2 \leq 2^j -1$ with $j \geq J_0$, the 2D wavelet functions are defined by
\[
    \psi^{\textint,k}_{j,(n_1,n_2)} := \begin{cases}
     \phiint_{j,n_1} \otimes \psiint_{j,n_2}, & j \geq J_0, k=1,\\
     \psiint_{j,n_1} \otimes \phiint_{j,n_2}, & j \geq J_0, k=2, \\
     \psiint_{j,n_1} \otimes \psiint_{j,n_2}, & j \geq J_0, k=3.
     \end{cases}
\]
Next, let
\[
    \Omega_1 =  \{ \phiint_{J_0,(n_1,n_2)} \, : \, 0 \leq n_1,n_2 \leq 2^{J_0}-1 \}
\]
and
\[
    \Omega_2 =  \{  \psi^{\textint,k}_{j,(n_1,n_2)}\, : \, j =J_0,1, \ldots, J-1, 0 \leq n_1,n_2\leq 2^j-1, k =1,2,3 \}.
\]
Then, for a fixed scale $J$, and for $N_J=2^{2J}$, the reconstruction space $\Rcal_{N_J}$ is given by
\beq
\label{eq:rec_space_b_wavelets}
    \Rcal_{N_J} = \spann\{ \varphi_i \, : \, \varphi_i \in \Omega_1 \cup \Omega_2, i = 1, \ldots, N_J\}.
\eeq
An ordering can be obtained analogously as in Subsection \ref{subsec:reconstructionspace}. 


\subsubsection{Sampling Space}

The sampling space can be chosen similar to Subsection \ref{subsec:samplingspace}, i.e., for $\varepsilon \leq 1$, we set
\[
    s_l^{(\varepsilon)} = \varepsilon e^{2 \pi i \varepsilon \langle l, \cdot \rangle} \cdot \chi_{[0,1]^2}, \quad l \in \Z^2,
\]
and for $M = (M_1,M_2 )\in \N \times \N$
\beq
\label{eq:samp_space_b_wavelets}
    \Scal^{(\varepsilon)}_{M} = \spann \left\{ s_l^{(\varepsilon)} \, : \, l = (l_1,l_2) \in \Z^2, - M_i \leq l_i \leq M_i , i =1,2 \right\}.
\eeq


\subsection{Main Result}

Our main result of this section states the linearity of the stable sampling rate for boundary wavelets, whose
proof is presented in Section \ref{sec:proof_b_wavelets}.

\begin{theorem}\label{theorem:boundarywavelets}
Let $J \in \N$, $\varepsilon \leq 1$, and $\theta >1$. Further, let $\Scal_M^{(\varepsilon)}$ and $\Rcal_{N_J}$ be defined as
\ref{eq:samp_space_b_wavelets} and \eqref{eq:rec_space_b_wavelets}, respectively. Then
\begin{align*}
    \inf \limits_{\substack{ \varphi \in \Rcal_{N_J} \\ \| \varphi\| =1 }} \| P_{S_M^{(\varepsilon)}} \varphi \| \geq \frac{1}{\theta}
\end{align*}
for $M = (\lceil S^\theta/\varepsilon\rceil 2^J,\lceil S^\theta /\varepsilon\rceil 2^J) \in \N \times \N$, where $S^\theta$ is a
constant independent of $J$.
In particular, the stable sampling rate obeys $\Theta(N_J,\theta) = \mathcal{O}(N_J)$ as $N_J \rightarrow \infty$ for every
fixed $\theta >1$
\end{theorem}

\begin{rem}
This result will be proved independently of Theorem \ref{maintheoremgeneral}, since boundary wavelets do
constitute an MRA but at highest scaling level, the space $V_J$ contains more than one generating function $\phi$.
It uses the reflected functions as well.
\end{rem}


\begin{figure}[h!]
        \begin{subfigure}[b]{0.5\textwidth}
        \includegraphics[scale=0.55]{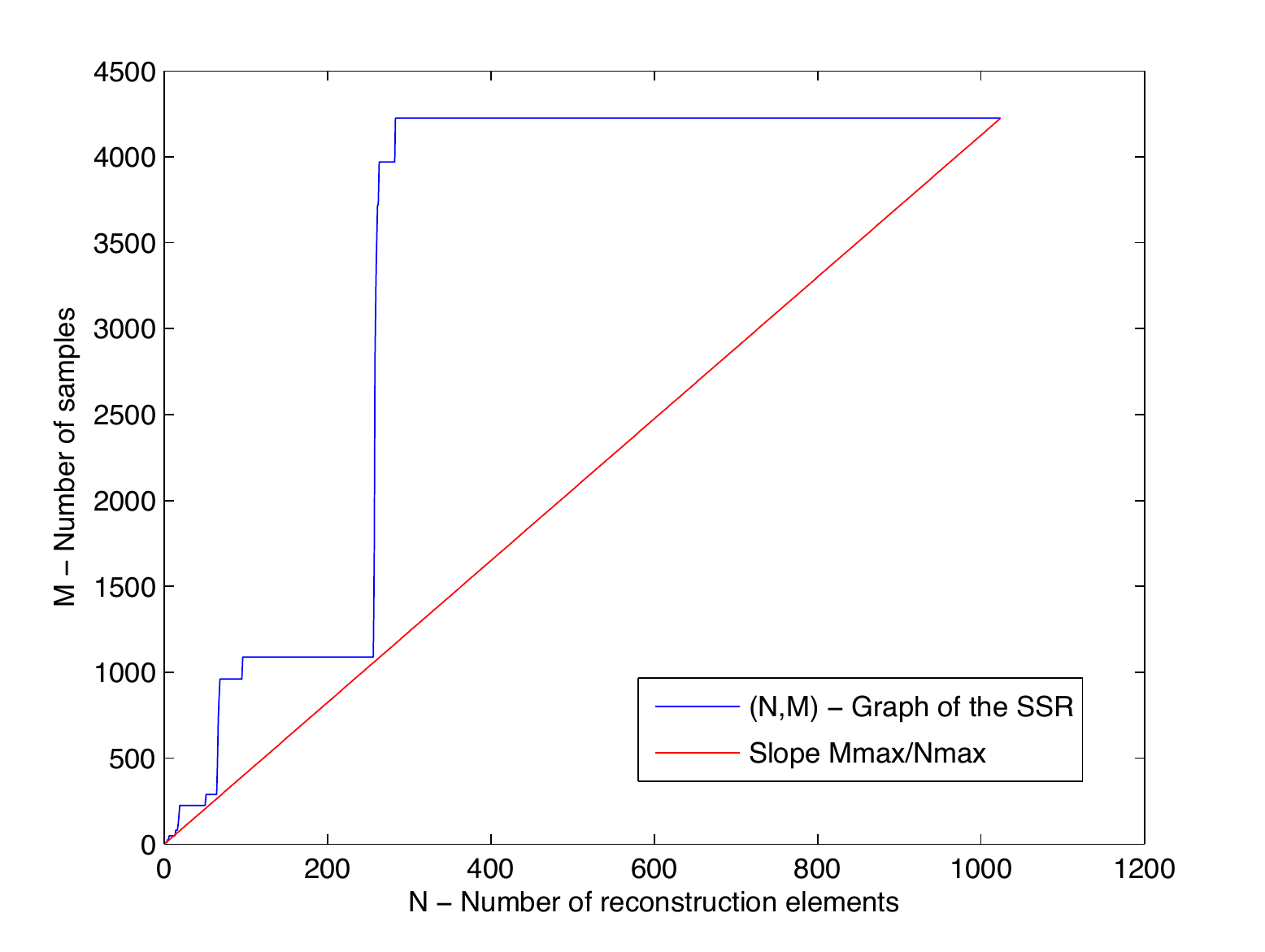}
        \caption{$\Theta(N,\theta)$ for Haar. Computed for  $J =4$ up to $N = 1024$ with $\varepsilon = 1/2$ and $\theta^{-1} = 0.45$.}\label{haar-1}
        \end{subfigure}\quad
    \begin{subfigure}[b]{0.5\textwidth}
        \includegraphics[scale=0.55]{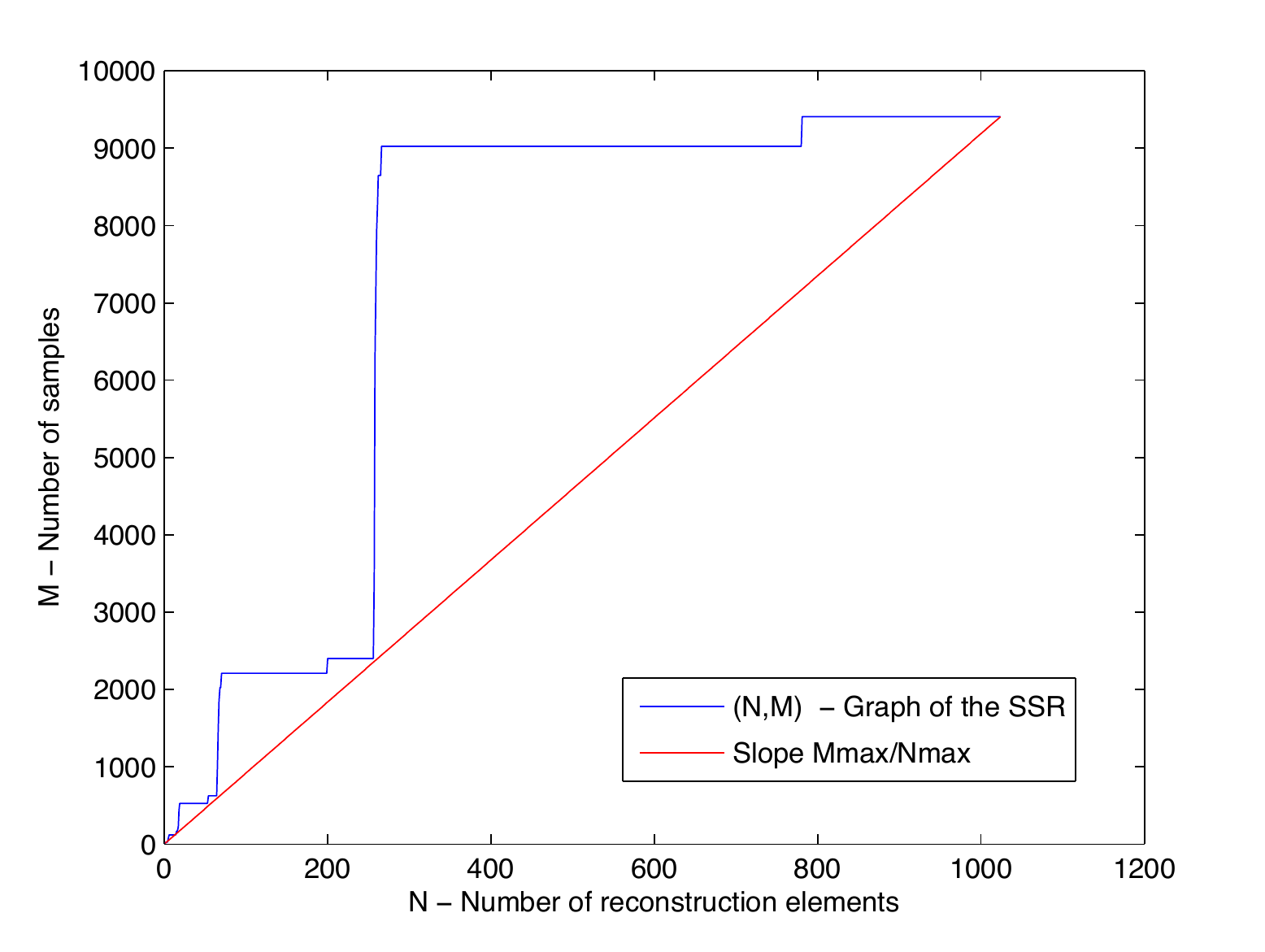}
        \caption{$\Theta(N,\theta)$ for Haar. Computed for  $J =4$ up to $N = 1024$ with $\varepsilon = 1/3$ and $\theta^{-1} = 0.45$.}\label{haar-2}
        \end{subfigure}
                \begin{subfigure}[b]{0.5\textwidth}
        \includegraphics[scale=0.55]{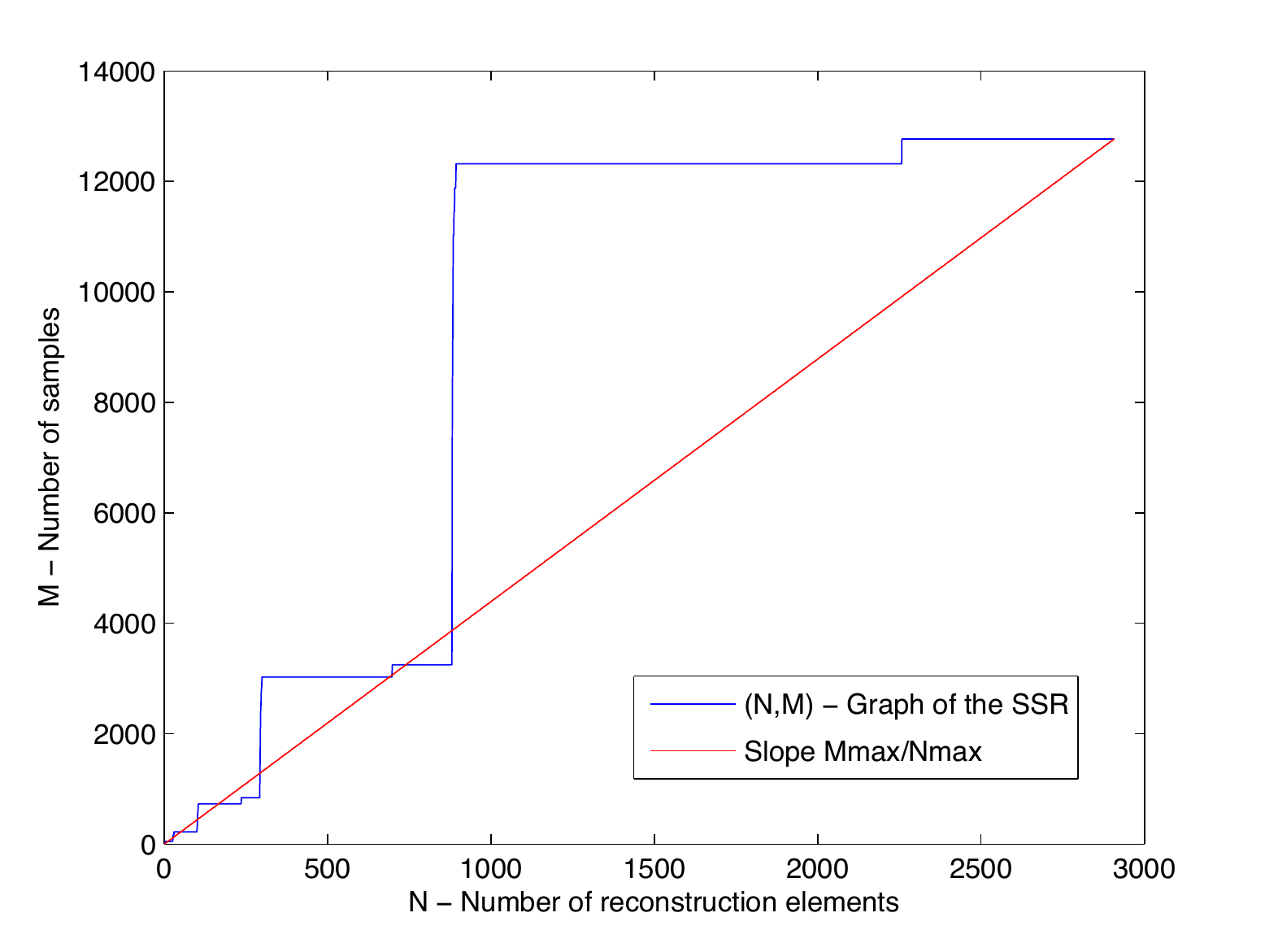}
        \caption{$\Theta(N,\theta)$ for Daubechies-4. Computed for  $J=3$ up to $N = 2908$ with $\varepsilon = 1/7$ and $\theta^{-1} = 0.45$.}\label{db4-1}
        \end{subfigure}\quad
        \begin{subfigure}[b]{0.5\textwidth}
            \includegraphics[scale=0.55]{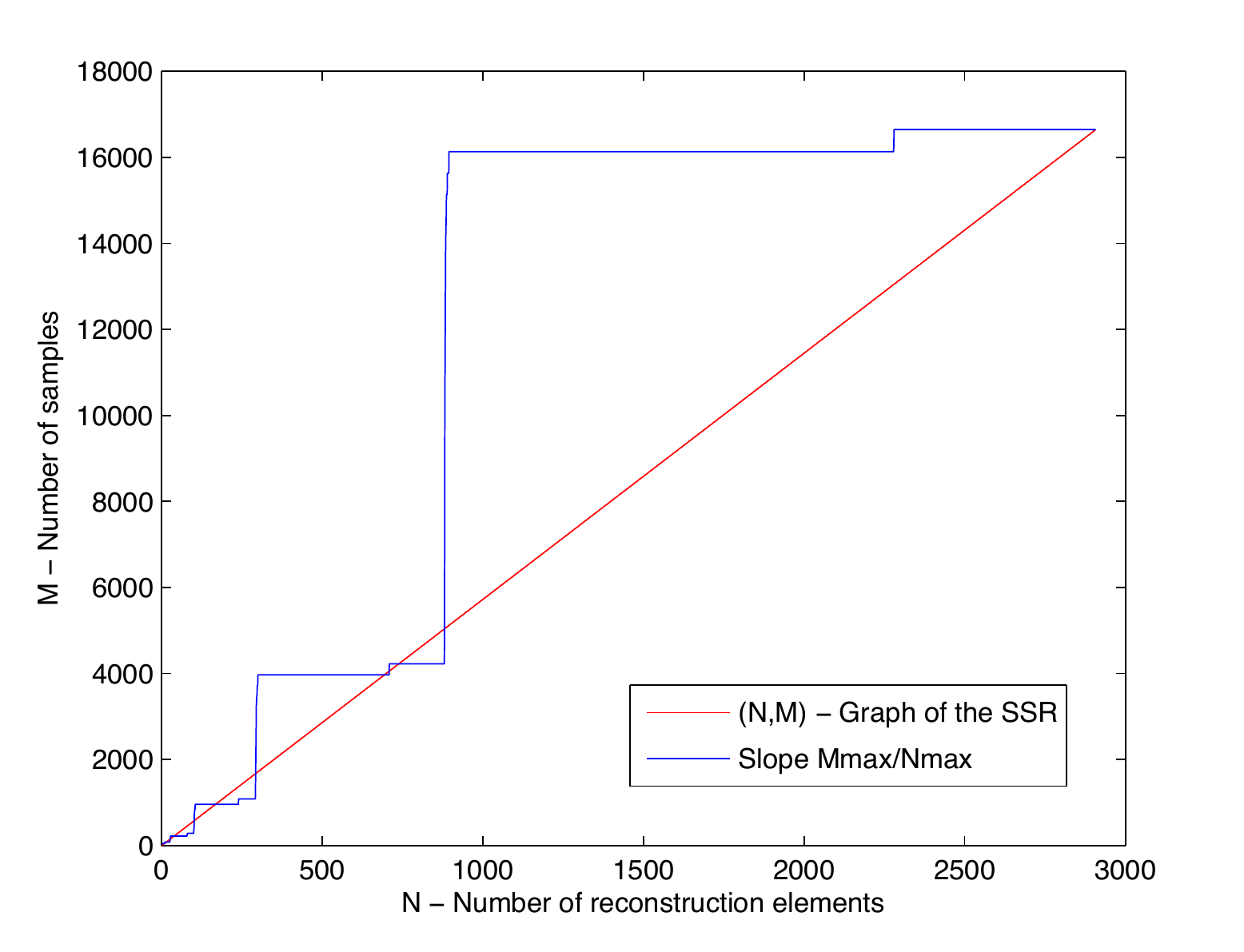}
        \caption{$\Theta(N,\theta)$ for Daubechies-4. Computed for  $J=3$ up to $N = 2908$ with $\varepsilon = 1/8$ and $\theta^{-1} = 0.45$.}\label{db4-2}
        \end{subfigure}
        \caption{Stable sampling rate $\Theta(N,\theta)$ for two dimensional dyadic Haar wavelets and two dimensional dyadic Daubechies-4 wavelets.}\label{plots1}
\end{figure}

\begin{table}
\begin{center}
\begin{tabular}{ll}
\begin{tabular}{@{}|c|c|c|c|}
\hline
\multicolumn{4}{|c|}{Numbers for Haar and  $\ \theta^{-1} = 0.45$}\\
\hline
$J$& $N_J$ & $M$ for $\varepsilon = 1/2$ & $M$ for $\varepsilon = 1/3$\\
\hline
0& 4 & 25 & 25\\
\hline
 1& 16 & 81 & 169\\
\hline
 2& 64 & 289 & 625\\
\hline
3& 256 &  1089 & 2401\\
\hline
4& 1024 & 4225 & 9409\\
\hline
\end{tabular}
&
\begin{tabular}{@{}|c|c|c|c|}
\hline
\multicolumn{4}{|c|}{Numbers for Daubechies and  $\ \theta^{-1} = 0.45$}\\
\hline
$J$& $N_J$ & $M$ for $\varepsilon = 1/7$ & $M$ for $\varepsilon = 1/8$\\
\hline
0& 100 & 225 & 289\\
\hline
 1& 292& 841 & 1089\\
\hline
 2& 880& 3249 & 4225\\
\hline
3& 2908&  12769 & 16641\\
\hline
4& 10408 &  - & -\\
\hline
\end{tabular}
\end{tabular}
\caption{Number of reconstruction elements and samples for Haar and Daubechies-4. Note that these numbers predict the jumps in Figure \ref{plots1}. }\label{tablenumber1}
\end{center}
\end{table}

\section{Numerical Experiments}
\label{sec:numerics}

In this section we numerically demonstrate the linearity of the stable sampling rate as stated in Theorem \ref{maintheoremgeneral}. We will also demonstrate how this combines with generalized sampling in practice. In particular, given this linearity, reconstructing from Fourier samples in smooth boundary wavelets will give an error decaying according to the smoothness and the number of vanishing moments.

In this section we consider dyadic scaling matrices
\begin{align}
    A = \begin{pmatrix} 2 & 0 \\ 0 & 2 \end{pmatrix}. \label{eq:diagonalmatrix}
\end{align}
Furthermore, our focus are \ti{separable wavelets}, i.e. wavelets that are obtained by tensor products of one dimensional scaling functions and one dimensional wavelet functions, respectively. Scaling matrices of the form (\ref{eq:diagonalmatrix}) preserve the separability.

\subsection{Linearity examples with Haar and Daubechies-4 wavelets}
We use the description of Section 3.1 and Example \ref{wavehaarexp} in order to perform numerical experiments for some known wavelets. This gives the reconstruction space $\Rcal = \overline{\spann} \{ \varphi \, : \, \varphi \in \Omega_1 \cup \Omega_2\}$ where $\Omega_1$ and $\Omega_2$ are defined in (\ref{Omega1}) and (\ref{Omega2}) respectively.
We order the reconstruction space $\Rcal$ in the same manner as presented at the end of Section 3. In (\ref{levelnumber}) we counted the number of reconstruction elements up to level $J-1$, which leads to
\begin{align}
N_J = (2^{2J}-1)a^2 + 6a(a-1)(2^J-1) + 3J(a-1)^2 + (2a-1)^2 \label{levelnumber2}
\end{align}
many elements, which is asymptotically of order $2^{2J}$.

We test Haar wavelets and  2D Daubechies-4 wavelets. Figure \ref{plots1} shows the linear behaviour of the stable sampling rate for these two types of wavelet generators. By a small abuse of notation, we also write $M$ for the total number of samples. In our analysis in Section 3 we estimated the angle $\cos(\omega(\Rcal_N, \Scal_M^{(\varepsilon)}))$ (recall that $\epsilon$ is the sampling rate) with respect to some fixed $\theta >1$. In fact we computed $M$ such that
\begin{align*}
    \inf \limits_{\substack{ \varphi \in \Rcal_N \\ \| \varphi\| =1}} \| P_{\Scal_M^{(\varepsilon)}} \varphi \| \geq \frac{1}{\theta}
\end{align*}
holds.
We proved that $M$ is up to a constant of the same size as $N_J$.  Figure \ref{plots1} shows the stable sampling rate (in blue)
\begin{align*}
    \Theta(N,\theta) = \min \left\{ M \in \N, \cos(\omega(\Rcal_N,\Scal_M)) \geq \frac{1}{\theta}\right\}
\end{align*}
and the linear function $f$ (in red) given by
\begin{align*}
    f(N) =  \frac{M_{\text{max}}}{N_{\text{max}}}N,
\end{align*}
where $N_{\text{max}}$ is the maximum value of $N$ used in the experiment and $M_{\text{max}} =  \Theta(N_{\text{max}},\theta)$.
We computed the stable sampling rate up to level $J = 4$. Note the (significant) jumps of the stable sampling rate occur whenever $N\in \N$ crosses the scaling level $N_J, J = 0, \ldots, 4$.  In the Haar case these are
\begin{align*}
    N_0 = 4, \qquad N_1 = 16,  \qquad N_2 = 64, \qquad N_3 = 256, \quad N_4 = 1024,
\end{align*}
see (\ref{levelnumber2}). Note that $a=1$ in the Haar case. In particular, the jumps are linear, suggesting a linear stable sampling rate. Figure 1 (b), (c), and (d) are interpreted similarly. However, the theoretical results are asymptotic results. Therefore, it should not be surprising that the stable sampling rate is below the linear line in some cases. It aligns asymptotically.

\subsection{Fourier samples and boundary wavelet reconstruction}
\begin{figure}
\begin{center}
\includegraphics[width=0.29\linewidth]{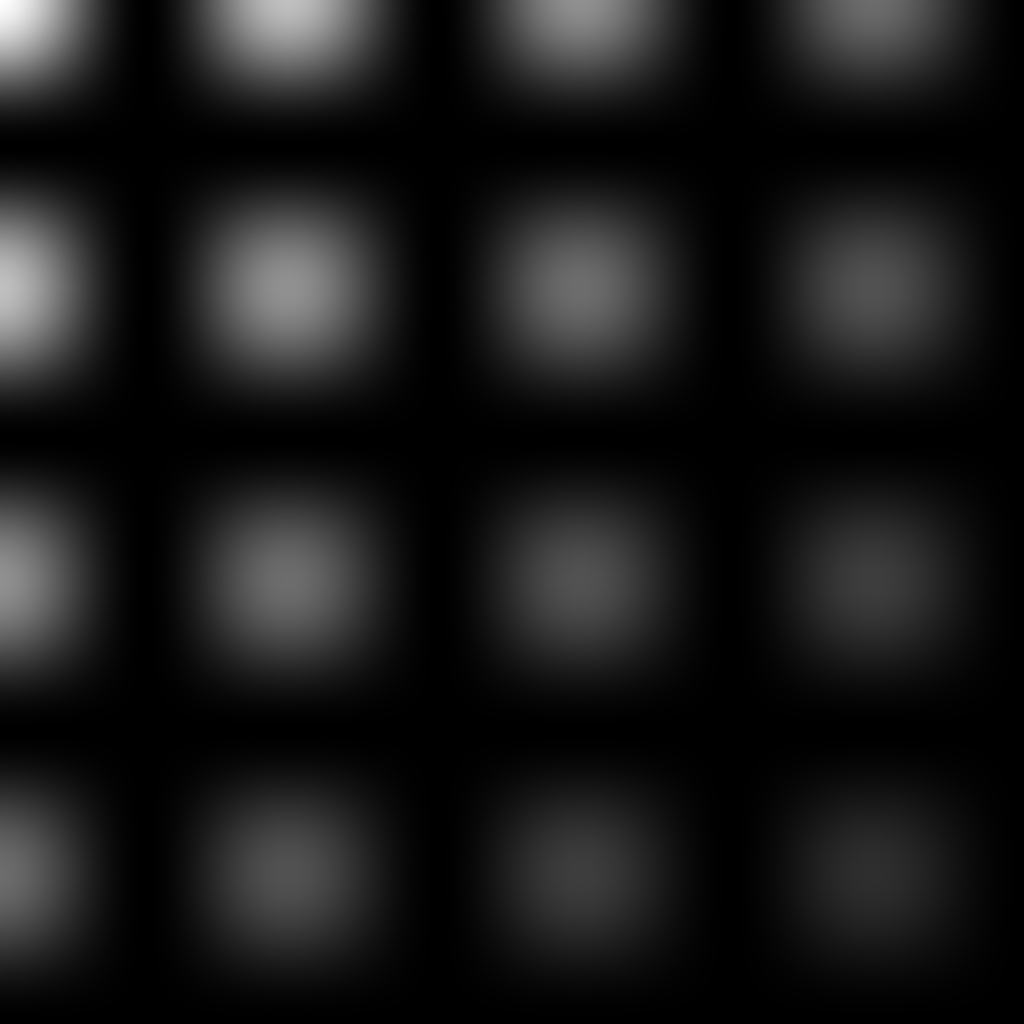}~~
\includegraphics[width=0.29\linewidth]{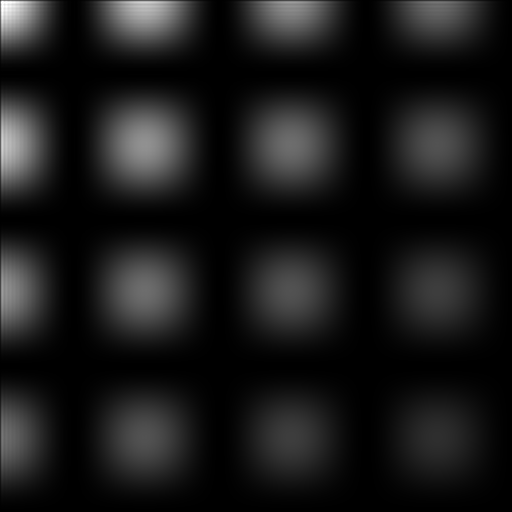}~~
\includegraphics[width=0.29\linewidth]{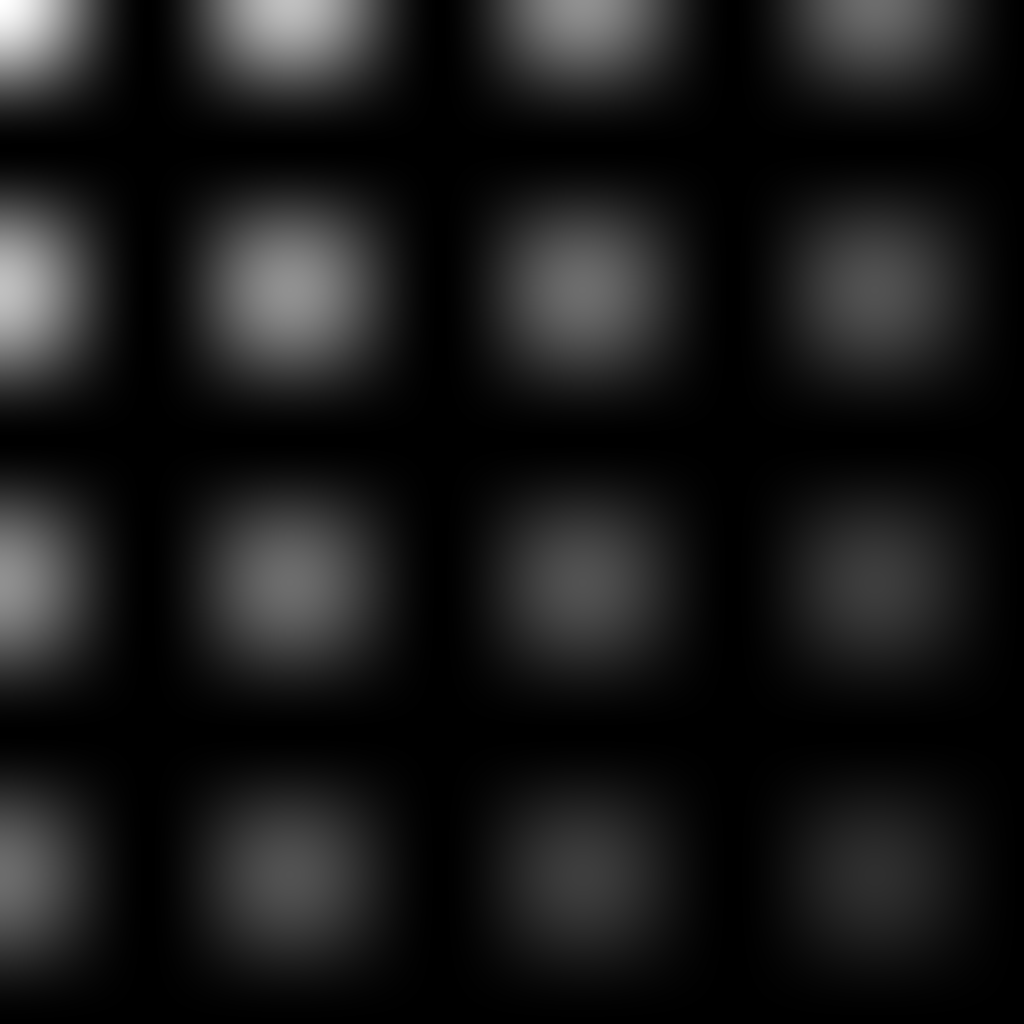}\\[8pt]
\includegraphics[width=0.29\linewidth]{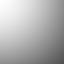}~~
\includegraphics[width=0.29\linewidth]{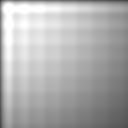}~~
\includegraphics[width=0.29\linewidth]{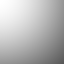}
\caption{Reconstruction of the function $f_1(x,y)=\cos^2(x)\exp(-y)$.
The second row shows an 8 times zoomed-in version of the upper left corner.
\textit{Left}: original function. \textit{Middle}: truncated Fourier series
with $256^2$ Fourier coefficients. \textit{Right}: generalized sampling with
Daubechies-3  wavelets computed from the same Fourier coefficients.}
\label{smooth2D_GS}
\end{center}
\end{figure}

\begin{figure}
\begin{center}
\includegraphics[width=0.495\linewidth]{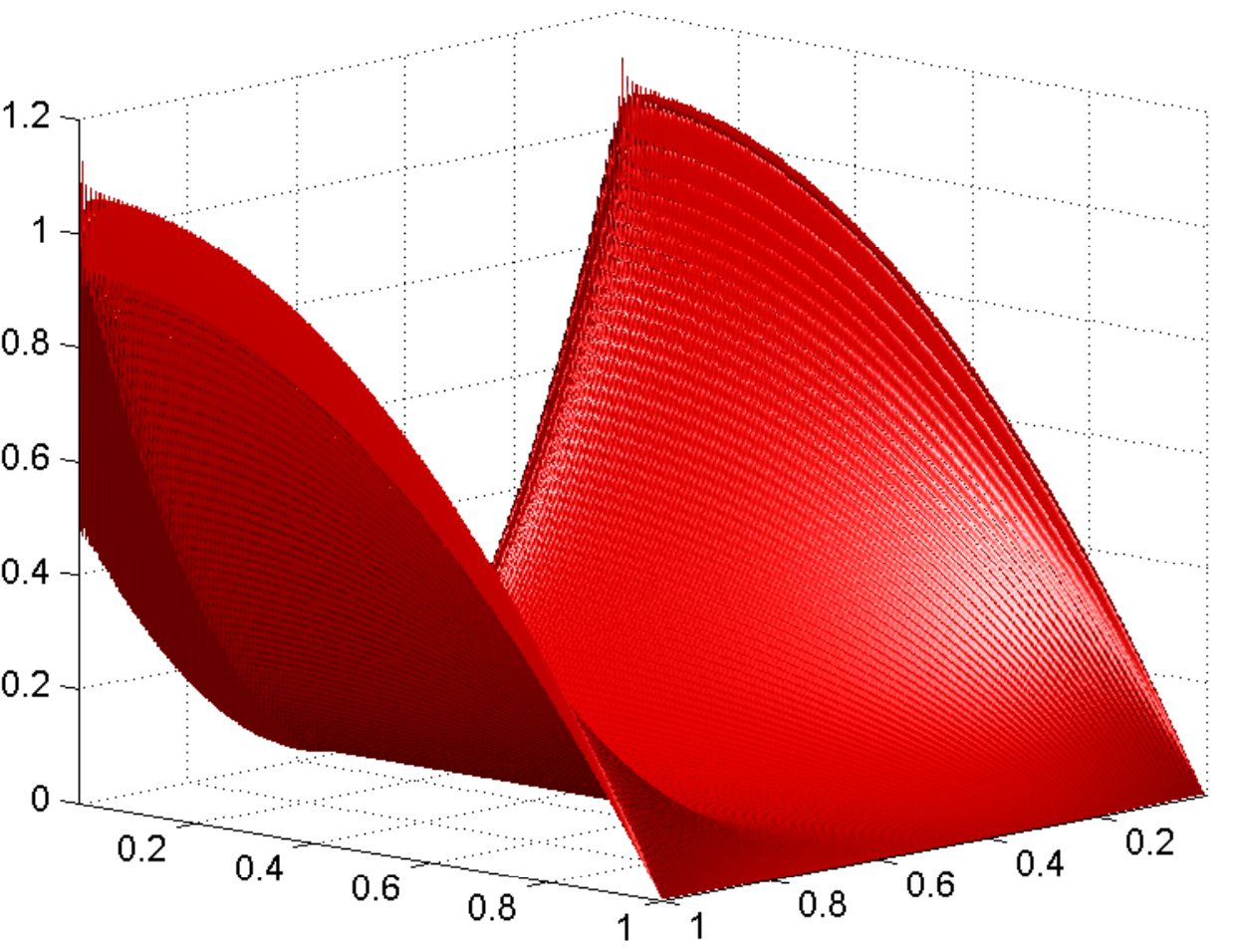}
\includegraphics[width=0.495\linewidth]{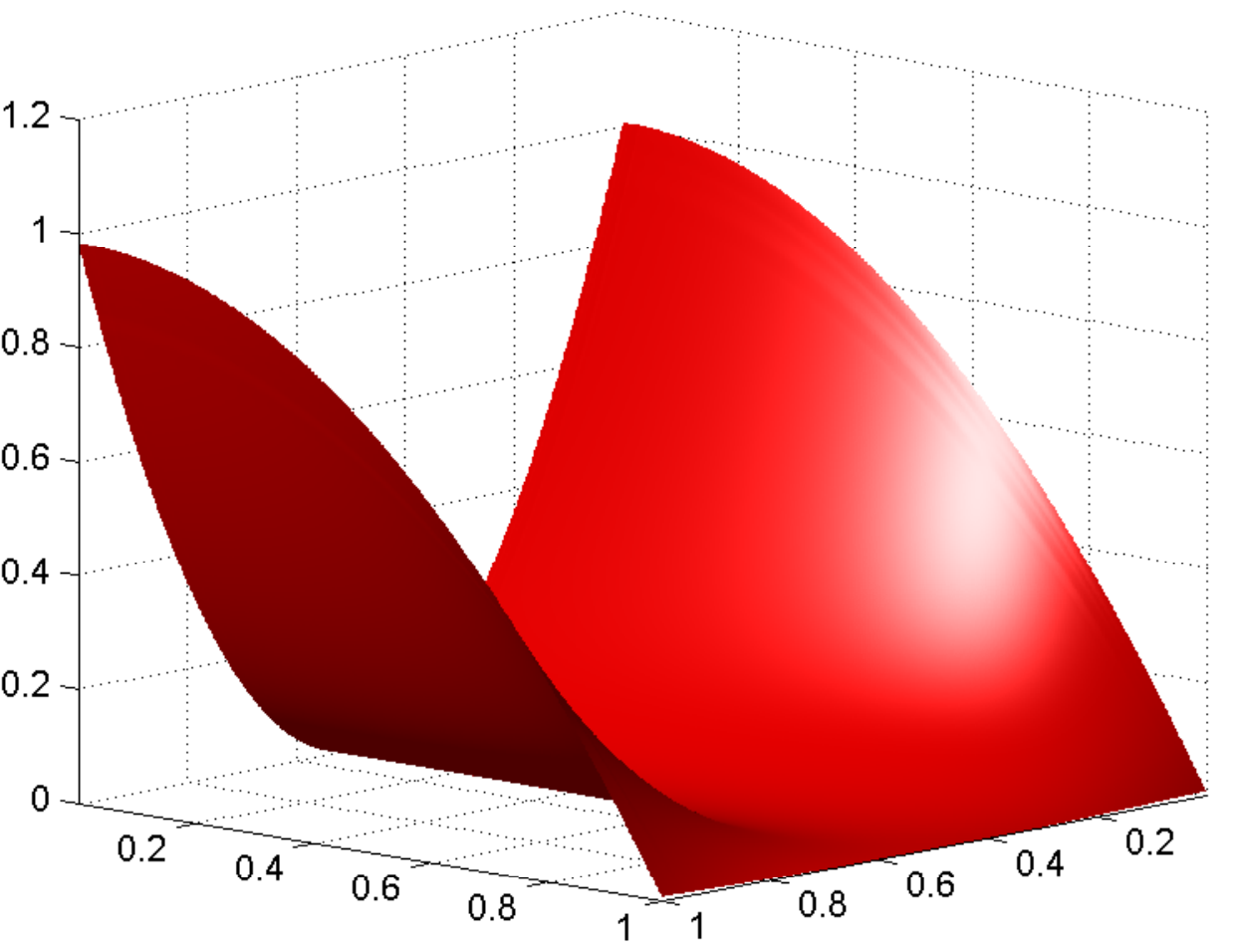}\\[5pt]
\includegraphics[width=0.495\linewidth]{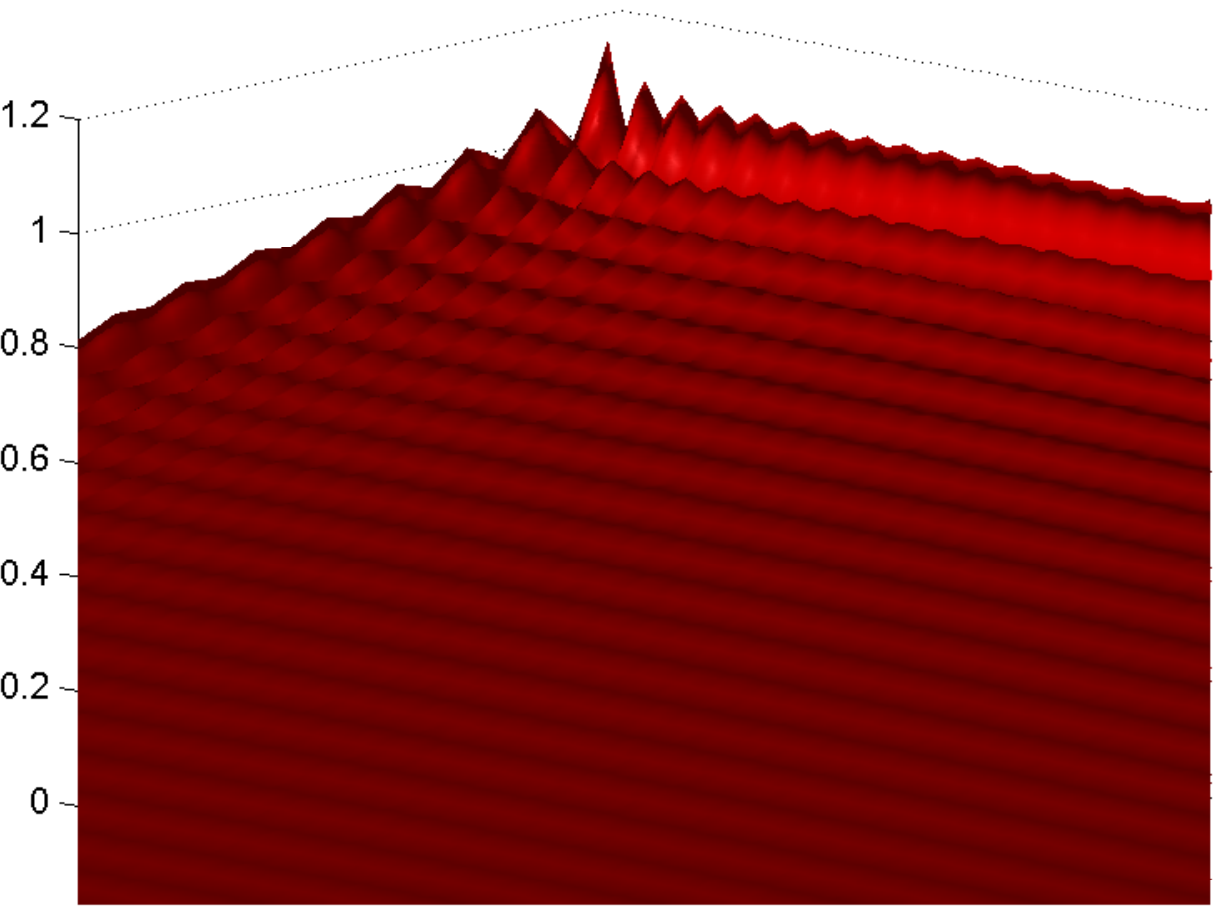}
\includegraphics[width=0.495\linewidth]{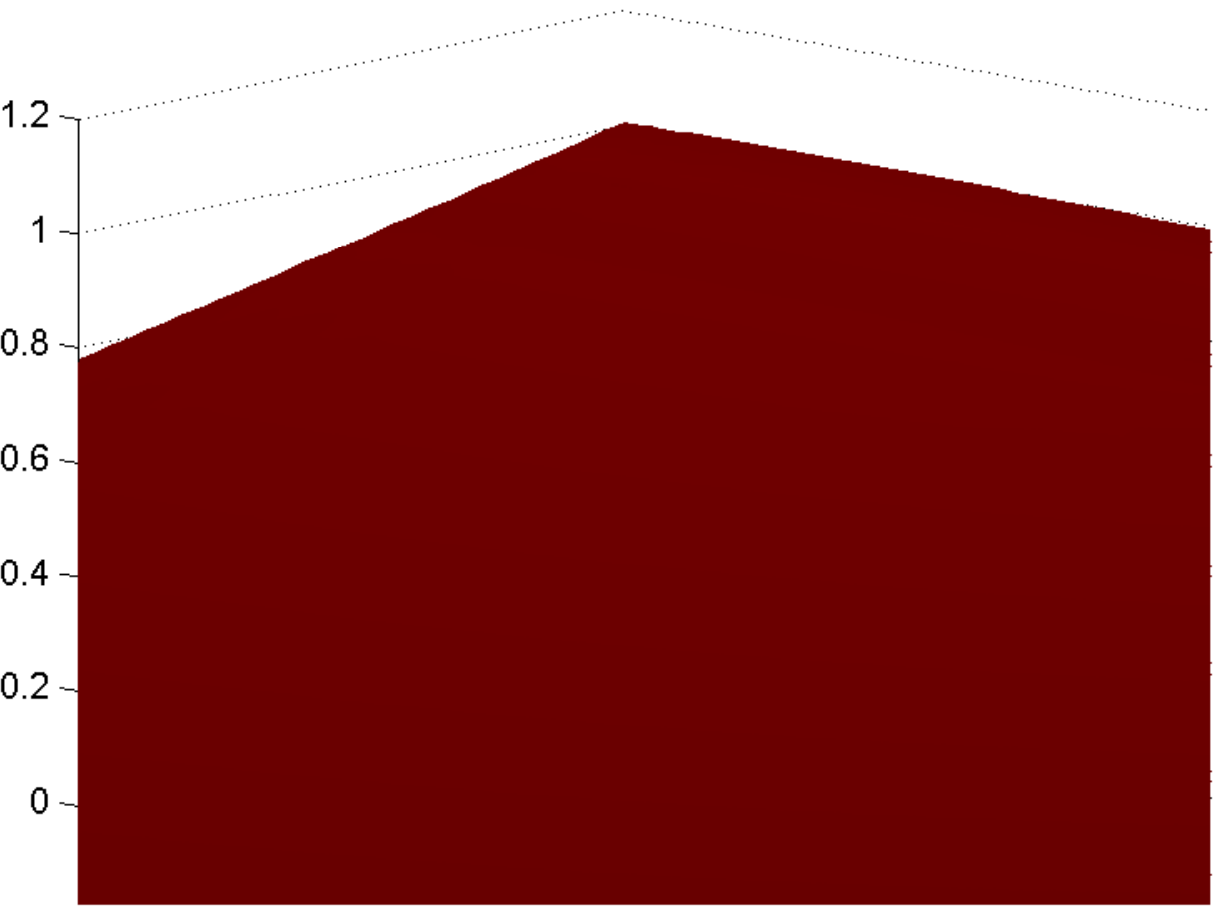}\\[5pt]
\includegraphics[width=0.495\linewidth]{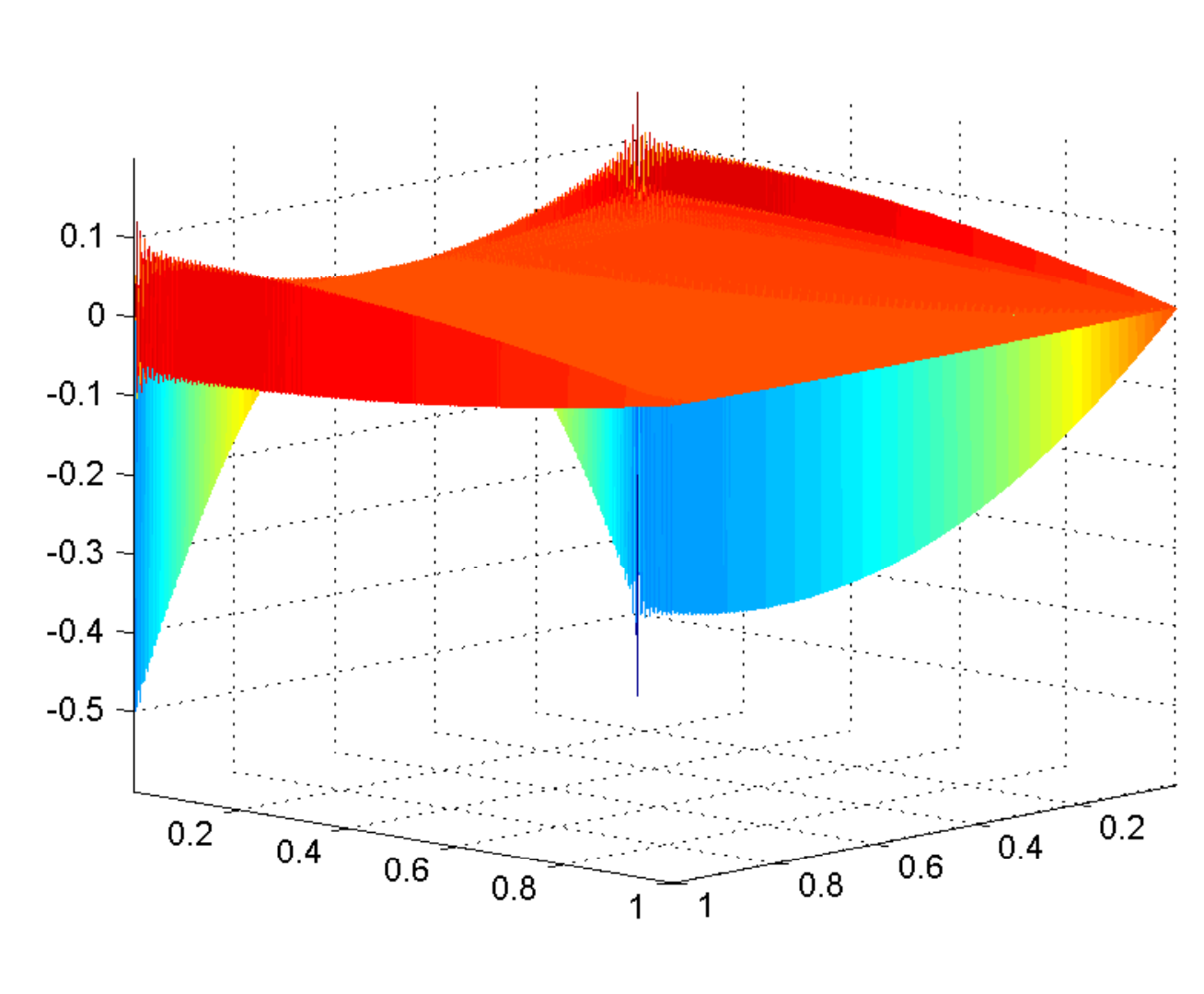}
\includegraphics[width=0.495\linewidth]{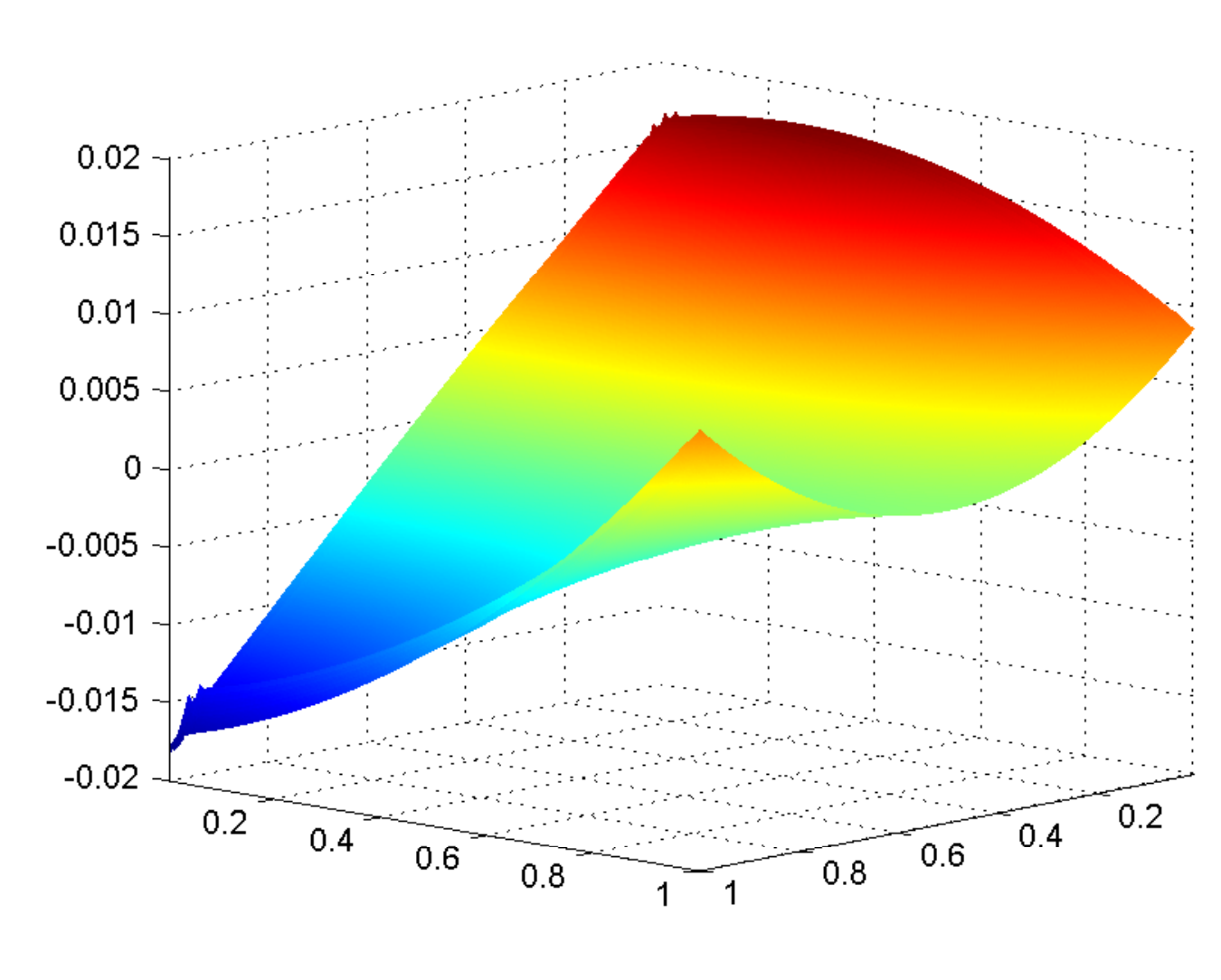}
\caption{Reconstruction of the function $f_2(x,y) = (1+x^2)(2y-1^2)$. \textit{Upper left}:
truncated Fourier series with $512^2$ Fourier coefficients. \textit{Middle
left}: 8 times zoomed-in version of the upper figure. \textit{Lower left}:
error committed by the truncated Fourier series.  \textit{Upper right}:
generalized sampling with Daubechies-3 wavelets computed from the same $512^2$ Fourier
coefficients. \textit{Middle right}: 8 times zoomed-in version of the upper
figure. \textit{Lower right}: error committed by generalized sampling.}
\label{smooth2D_GS_xy}
\end{center}
\end{figure}

In this example we will demonstrate the efficiency of generalized sampling given the established linearity of the stable sampling rate. In particular, suppose that $f$ is a function we want to recover from its Fourier information. It is smooth, however, not periodic -- a problem that occurs for example in electron microscopy and also in MRI. This causes the classical Fourier reconstruction to converge slowly, yet a smooth boundary wavelet basis will give much faster convergence (see \cite[Ch.\ 9]{Mallat}).  As discussed, the issue is that we are given Fourier samples, not wavelet coefficients.  However, this is not a problem in view of the linearity of the stable sampling rate. In particular, if $f \in W^s(0,1)$, where $\mathrm{W}^s(0,1)$ denotes the usual Sobolev space, and $P_{\Rcal_N}$ denotes the projection onto the space $\Rcal_{N}$ of the first $N$ boundary wavelets (see (\ref{eq:rec_space_b_wavelets})), then
$$
 \| f - P_{\Rcal_N} f \| = \mathcal{O}(N^{-s}),\quad N \rightarrow \infty,
$$
provided that the wavelet has sufficiently many vanishing moments.
Now, if $G_{N,M}(f) \in \Rcal_{N}$ is the generalized sampling solution from Definition \ref{generalized_sampling} given $M$ Fourier coefficients, and $M$ is
chosen according to the stable sampling rate then
$$
\| f - G_{N,M}(f)\| = \mathcal{O}(N^{-s}) = \mathcal{O}(M^{-s}),\quad N \rightarrow \infty.
$$
Hence, we obtain the same convergence rate up to a constant, by simply postprocessing the given samples.

To illustrate this advantage we will consider the following two functions:
$$
f_1(x,y)=\cos^2(x)\exp(-y), \qquad f_2(x,y) = (1+x^2)(2y-1^2).
$$
In Figure \ref{smooth2D_GS} we have shown the results for $f_1$ and compared the classical Fourier reconstruction with the generalized sampling reconstruction. Both examples use exactly the same samples, however, note the pleasant absence of the Gibb's ringing in the generalized sampling case. The same experiment is carried out for $f_2$ in Figure
\ref{smooth2D_GS_xy}, however, here we have displayed the reconstructions in $3D$ in order to visualize the error.


\section{Proof of Theorem \ref{maintheoremgeneral}}
\label{sec:proofWavelets}

The proof of Theorem \ref{maintheoremgeneral} is somewhat technical, wherefore we divide the proof into several steps.
First, in Subsection \ref{subsec:wavelets_structure_proof}, the overall structure of the proof is presented, and
the respective details can then be found in Subsection \ref{subsec:wavelets_structure_proof}.


\subsection{Structure of the Proof}
\label{subsec:wavelets_structure_proof}

Let $\varepsilon \in (0,1/(T_1+T_2)]$ and $\theta >1$. Then we have to prove that
\beq \label{eq:mainproof1_claim}
    \inf \limits_{\substack{ \varphi \in \Rcal_{N_J} \\ \|\varphi\| =1}} \| P_{\Scal_M^{(\varepsilon)}} \varphi \| \geq \frac{1}{\theta},
\eeq
for
\[
M=(M_1,M_2) = \left(\frac{ \lambda_1^{(J)} +  \lambda_3^{(J)}}{\varepsilon} S^{(\theta)}, \frac{ \lambda_2^{(J)} +  \lambda_4^{(J)}}{\varepsilon} S^{(\theta)}\right),
\]
and some $S^{(\theta)}$ independent of $J$. To this end, let $\varphi \in \Rcal_N$ be such that $\| \varphi \|=1$. Since the sampling functions
form an orthonormal basis of $\Scal_M^{(\varepsilon)}$, we obtain
\beq \label{eq:mainproof1_Pphi}
\| P_{\Scal_M^{(\varepsilon)}} \varphi \|^2= \sum \limits_{l_1 = -M_1}^{M_1} \sum \limits_{l_2 = -M_2}^{M_2}
|\langle \varphi , s_l^{(\varepsilon)} \rangle|^2, \quad l = (l_1,l_2).
\eeq
 By Lemma \ref{lemma:V_0plusW_j}, which relies mainly on the underlying MRA structure, we can write $\varphi$ in terms
 of scaling functions at highest scale, i.e.,
\[
    \varphi =  \sum \limits_{l_1 = L_3}^{L_1} \sum \limits_{l_2 = L_4}^{L_2} \alpha_{l_1,l_2}\phi_{J,(l_1,l_2)},
\]
for some $L_1,L_2,L_3,L_4 \in \Z$.
Lemma \ref{lemma:phis_l}, which is proven by direct computations, then shows that
\[
    \langle \varphi, s_l^{(\varepsilon)} \rangle = \varepsilon |\det A|^{-J/2} \Phi (\varepsilon(A^{-J})^Tl) \phihat(\varepsilon(A^{-J})^Tl),
\]
where $\Phi$ is a trigonometric polynomial of the form
\[
    \Phi(z) =  \sum \limits_{m_1 = L_3}^{L_1} \sum \limits_{m_2 = L_4}^{L_2} \alpha_{m_1,m_2} e^{-2\pi i  \langle z, m \rangle}, \quad z \in \R^2, m = (m_1,m_2).
\]
Using \eqref{eq:mainproof1_Pphi}, we conclude that
\[
\|P_{\Scal_M^{(\varepsilon)}} \varphi \|^2= \sum \limits_{l_1 = -M_1}^{M_1} \sum \limits_{l_2 = -M_2}^{M_2} \varepsilon^2 |\det A|^{-J}
|\Phi (\varepsilon(A^{-J})^Tl)|^2 |\phihat(\varepsilon(A^{-J})^Tl)|^2, \quad l = (l_1, l_2).
\]
By Theorem \ref{maintheoremdelta}, there exist some $\varepsilon_0 > 0$ and $S^{(\theta)} = \left(S^{(\theta)}_1,S^{(\theta)}_2\right)\in \N \times \N$
that does not depend on $J$ such that
\[
    \sum \limits_{l_1 = -M_1}^{M_1} \sum \limits_{l_2 = -M_2}^{M_2} \varepsilon_0 ^2 | \det A^{-J}| | \Phi(\varepsilon_0 (A^{-J})^T l) |^2
    | \phihat(\varepsilon_0(A^{-J})^Tl)|^2 \geq  \frac{\|\Phi\|^2}{\theta^2}
\]
for
\[
 M= \frac{1}{\varepsilon_0} (A^J)^T \widetilde{S}^{(\theta)}=
 \begin{pmatrix} \frac{ \lambda_1^{(J)} + \lambda_3^{(J)}}{\varepsilon_0} S^{(\theta)}_1 \\
 \frac{ \lambda_2^{(J)} +  \lambda_4^{(J)}}{\varepsilon_0} S^{(\theta)}_2 \end{pmatrix}.
 \]
Since
\[
\| \Phi \|^2 =\sum \limits_{m_1 = L_3}^{L_1} \sum \limits_{m_2 = -L_4}^{L_2} |\beta_{m_1,m_2}|^2 =\| \varphi\|^2 = 1,
\]
we obtain
\[
\| P_{\Scal_M^{(\varepsilon_0)}} \varphi \|^2 \geq \frac{1}{\theta^2}.
\]
Finally, Lemma \ref{lemma:oneepsenough} implies that a change of $\varepsilon$ only changes the constant, showing that
\eqref{eq:mainproof1_claim} is true for any $\varepsilon \in (0,1/(T_1+T_2)]$, thereby proving Theorem
\ref{maintheoremgeneral}.


\subsection{Auxiliary results}\label{subsec:wavelets_proofs}

In this section we will prove the results mentioned in Subsection \ref{subsec:wavelets_structure_proof}, which
are required for completing the proof of Theorem \ref{maintheoremgeneral}. The following result is an extension of a result from \cite{AHP2} to the two dimensional setting.

\begin{lem}\label{lemma:V_0plusW_j}
Let $J \in \N$ and $\varphi \in V_0^{(a)} \oplus W_0^{(a)} \oplus \ldots \oplus W_{J-1}^{(a)}$. Then there exist
$L_1,L_2,L_3,L_4 \in \Z$ dependent of $J$ such that
\[
    \varphi =  \sum \limits_{l_1 = L_3}^{L_1} \sum \limits_{l_2 = L_4}^{L_2} \alpha_{l_1,l_2}\phi_{J,(l_1,l_2)}.
\]
\end{lem}

\begin{proof}
Let $\varphi \in V_0^{(a)} \oplus W_0^{(a)} \oplus \ldots \oplus W_{J-1}^{(a)}$. Then $\varphi$ has an expansion of the following form
\beq \label{eq:proof1_V0W0}
    \varphi = \sum _{m_1 = - a +1}^{m_1 = a-1} \sum _{m_2 = - a +1}^{m_2 = a-1} \alpha_{m_1,m_2} \phi_{0,(m_1,m_2)} +
    \sum _{p=1}^{|\det A| -1} \sum_{j=0}^{J-1} \sum_{m_1 = -a +1}^{a(\lambda_1^{j}+ \lambda_{2}^{j})-1}
    \sum \limits_{m_2= -a +1}^{a(\lambda_3^{j}+ \lambda_{4}^{j})-1} \beta_{j,(m_1,m_2)}^p \psi_{j,(m_1,m_2)}^p.
\eeq
Since  $V_0^{(a)} \subset V_0$ and $W_j^{(a)} \subset W_{j}$ for  $j \in \N$ and the sequence $(V_j)_{j\in\Z}$ forms an MRA,
it follows that
\[
V_0^{(a)} \oplus W_0^{(a)} \oplus \ldots \oplus W_{J-1}^{(a)} \subset V_J.
\]
Since an orthonormal basis for $V_J$ is given by the functions $\{ \phi_{J,m} \, : \, m \in \Z^2\}$, for each $|m_i| < a, i =1,2$, we have
\[
    \phi_{0,(m_1,m_2)} = \sum _{l \in \Z^2} \langle \phi_{0,(m_1,m_2)}, \phi_{J,l} \rangle \phi_{J,l}
\]
Moreover, since $\phi_{0,(m_1,m_2)}$ and $\phi_{J,l}$ are compactly supported, we obtain
\[
    \langle \phi_{0,(m_1,m_2)}, \phi_{J,l} \rangle \neq 0,
\]
 if
\[
 -a+   (-a+1)( \lambda_1^{(J)}+ \lambda_2^{(J)} ) \leq l_1 \leq (2a-1)(\lambda_1^{(J)} + \lambda_2^{(J)})
\]
and
\[
 -a+   (-a+1)( \lambda_3^{(J)}+ \lambda_4^{(J)} ) \leq l_2 \leq (2a-1)(\lambda_3^{(J)} + \lambda_4^{(J)}).
\]
This follows by a straightforward computation from the support conditions of $\phi_{0,(m_1,m_2)}$ and $\phi_{J,l}$ together with
 $|m_i| < a, i =1,2$.
Similarly, we have  $\langle \psi_{j,(m_1,m_2)}^p, \phi_{J,(l_1,l_2)} \rangle \neq 0$ if

\textbf{Case \textrm{I}:} $\det A^j \geq 0$
 \begin{align*}
-a&  - \lambda_1^{(J)}\frac{(-a+1)\lambda_4^{(j)}+\lambda_2^{(j)}(a(\lambda_3^{(j)}+\lambda_4^{(j)})-1)}{\det A^j} + \lambda_2^{(J)} \frac{(-a+1)\lambda_1^{(j)}+(a(\lambda_1^{(j)}+\lambda_2^{(j)})-1)\lambda_3^{(j)}}{\det A^j} \\
& < l_1 < \lambda_1^{(J)} \frac{a(\lambda_4^{(j)}-\lambda_2^{(j)})+(a(\lambda_1^{(j)}+\lambda_2^{(j)})-1)\lambda_4^{(j)}+(a-1)\lambda_2^{(j)}}{\det A^j}+\\&\qquad +\lambda_2^{(J)} \frac{a(\lambda_1^{(j)}-\lambda_3^{(j)})+(a(\lambda_3^{(j)}+\lambda_4^{(j)})-1)\lambda_1^{(j)}+(a-1)\lambda_3^{(j)}}{\det A^j}
 \end{align*}
 and
  \begin{align*}
-a&  - \lambda_3^{(J)}\frac{(-a+1)\lambda_4^{(j)}+\lambda_2^{(j)}(a(\lambda_3^{(j)}+\lambda_4^{(j)})-1)}{\det A^j} + \lambda_4^{(J)} \frac{(-a+1)\lambda_1^{(j)}+(a(\lambda_1^{(j)}+\lambda_2^{(j)})-1)\lambda_3^{(j)}}{\det A^j} \\
& < l_2 < \lambda_3^{(J)} \frac{a(\lambda_4^{(j)}-\lambda_2^{(j)})+(a(\lambda_1^{(j)}+\lambda_2^{(j)})-1)\lambda_4^{(j)}+(a-1)\lambda_2^{(j)}}{\det A^j}+\\&\qquad +\lambda_4^{(J)} \frac{a(\lambda_1^{(j)}-\lambda_3^{(j)})+(a(\lambda_3^{(j)}+\lambda_4^{(j)})-1)\lambda_1^{(j)}+(a-1)\lambda_3^{(j)}}{\det A^j}.
 \end{align*}

\textbf{Case \textrm{II}:} $\det A^j < 0$
 \begin{align*}
-a&  - \lambda_1^{(J)}\frac{(-a+1)\lambda_4^{(j)}+\lambda_2^{(j)}(a(\lambda_3^{(j)}+\lambda_4^{(j)})-1)}{\det A^j} + \lambda_2^{(J)} \frac{(-a+1)\lambda_1^{(j)}+(a(\lambda_1^{(j)}+\lambda_2^{(j)})-1)\lambda_3^{(j)}}{\det A^j} \\
& > l_1 > \lambda_1^{(J)} \frac{a(\lambda_4^{(j)}-\lambda_2^{(j)})+(a(\lambda_1^{(j)}+\lambda_2^{(j)})-1)\lambda_4^{(j)}+(a-1)\lambda_2^{(j)}}{\det A^j}+\\&\qquad +\lambda_2^{(J)} \frac{a(\lambda_1^{(j)}-\lambda_3^{(j)})+(a(\lambda_3^{(j)}+\lambda_4^{(j)})-1)\lambda_1^{(j)}+(a-1)\lambda_3^{(j)}}{\det A^j}
 \end{align*}
 and
  \begin{align*}
-a&  - \lambda_3^{(J)}\frac{(-a+1)\lambda_4^{(j)}+\lambda_2^{(j)}(a(\lambda_3^{(j)}+\lambda_4^{(j)})-1)}{\det A^j} + \lambda_4^{(J)} \frac{(-a+1)\lambda_1^{(j)}+(a(\lambda_1^{(j)}+\lambda_2^{(j)})-1)\lambda_3^{(j)}}{\det A^j} \\
& > l_2 > \lambda_3^{(J)} \frac{a(\lambda_4^{(j)}-\lambda_2^{(j)})+(a(\lambda_1^{(j)}+\lambda_2^{(j)})-1)\lambda_4^{(j)}+(a-1)\lambda_2^{(j)}}{\det A^j}+\\&\qquad +\lambda_4^{(J)} \frac{a(\lambda_1^{(j)}-\lambda_3^{(j)})+(a(\lambda_3^{(j)}+\lambda_4^{(j)})-1)\lambda_1^{(j)}+(a-1)\lambda_3^{(j)}}{\det A^j}.
\end{align*}
 Minimizing the lower bounds with respect to $j \in \{ 0, \ldots, J-1\}$ and maximizing the upper bounds with respect to $j$, respectively, yields the claim.
 \end{proof}

The following lemma is well known (see \cite{Mallat}).
\begin{lem}\label{lemma:orthosumone}
Let $f \in L^2(\R^2)$. Then $\{ f( \cdot - m) \, : \, m \in \Z^2\}$ is an orthonormal system if and only if
\begin{align*}
	\sum \limits_{m \in \Z^2} | \fhat(\xi +  m )|^2 = 1 \quad \text{for almost every } \xi \in \R^2.
\end{align*}
\end{lem}
Finally, one more technical lemma is needed.
\begin{lem}\label{lemma:phis_l}
Let $A$ be a scaling matrix, $J \in \Z$, and $m=(m_1,m_2) \in \Z^2$. Further, let $\varphi \in \overline{\spann} \{ \phi_{J,m} \, : \, m \in \Z^2\}$
be compactly supported in $[-T_1,T_2]^2$, and let $L_1,L_2,L_3,L_4 \in \Z$ be such that
\[
    \varphi = \sum \limits_{m_1 = L_3}^{L_1} \sum \limits_{m_2 = L_4}^{L_2} \alpha_{m_1,m_2} \phi_{J,(m_1,m_2)}, \quad \alpha_m \in \C.
\]
Then, for all $l \in \Z^2$,
\[
    \langle \varphi, s_l^{(\varepsilon)} \rangle = \varepsilon |\det A|^{-J/2} \Phi (\varepsilon(A^{-J})^Tl) \phihat(\varepsilon(A^{-J})^Tl),
\]
where $s_l^{(\varepsilon)}$ is defined in (\ref{s_l}) and $\Phi$ is the trigonometric polynomial given by
\[
    \Phi(z) = \sum \limits_{\substack{ L_3 \leq m_1 \leq L_1, \\ L_4 \leq m_2 \leq L_2,}} \alpha_m e^{-2\pi i  \langle z, m \rangle}, \quad z \in \R^2, m = (m_1,m_2).
\]
\end{lem}

\begin{proof}
Since $\varphi$ is supported in $[-T_1,T_2]^2$, we obtain
\begin{align*}
    \langle \varphi, s_l^{(\varepsilon)} \rangle &= \varepsilon \int \limits_{\R^2} \varphi(x) e^{-2 \pi i \varepsilon \langle l, x \rangle}
    \cdot \chi_{[-T_1,T_2]^2} \, dx \\
    &= \varepsilon \widehat{\varphi}( \varepsilon l) \\
    &= \varepsilon  \left(\sum \limits_{m_1 = L_1}^{L_2} \sum \limits_{m_2 = L_3}^{L_4} \alpha_{m_1,m_2} \phi_{J,(m_1,m_2)}\right)^{\wedge}(\varepsilon l) \\
    &= \varepsilon  \sum \limits_{m_1 = L_1}^{L_2} \sum \limits_{m_2 = L_3}^{L_4} \alpha_{m_1,m_2} \left( \phi_{J,(m_1,m_2)}\right)^{\wedge}(\varepsilon l)\\
    &= \varepsilon |\det A|^{-J/2} \Phi (\varepsilon(A^{-J})^T l) \phihat(\varepsilon(A^{-J})^Tl).
\end{align*}
This proves the claim.
\end{proof}


\subsection{Theorem \ref{maintheoremdelta}}

The proof of Theorem \ref{maintheoremdelta} requires a particular estimate (Proposition \ref{pottstheorem}) for the norm of
trigonometric polynomials depending on their evaluations on a particular grid whose mesh norm and associated Voronoi
regions come also into play.


\subsubsection{Mesh Norm}

We start with the definition of a mesh norm for the situation we are faced with. A mesh norm can be interpreted as the
largest distance between neighboring nodes.

\begin{deff}
\label{defi:meshes}
Let $\Lambda\subset \Z^2$ be an integer grid of the form
\[
\Lambda := \left\{ (l_1,l_2) \in \Z^2 \, : \, -M_i \leq l_i \leq M_i, i=1,2 \right\}, \quad M_1,M_2 \in \N,
\]
and let $A$ be a scaling matrix. Set
\[
\Omega := \overline{\Lambda}^A :=  A ([-M_1,M_1] \times [-M_2,M_2]) \subset\R^2,
\]
and define a metric $\rho$ on $\Omega$ by
\[
    \rho: \Omega \times \Omega \longrightarrow \R^+, \quad (x,y) \mapsto \min \limits_{k \in A(\Z^2)} \| x - y + k\|_\infty.
\]
The \emph{mesh norm} of $\{ x_l \in \Omega \, : \, l \in \Lambda\}$ is then defined as
\[
    \delta := \max \limits_{x \in \Omega} \min \limits_{l \in \Lambda} \rho(x_l,x),
\]
where $x_l := A \cdot l, l \in \Lambda$ denote the \emph{nodes} in $\Omega$.
\end{deff}

Before we can continue, we require some notions and results on Voronoi regions and trigonometric polynomials. Our first result
shows that if the distance between the nodes $\{x_l\}_l$ converges to zero, the mesh norm of the entire grid $\Omega$ converges
to zero.

\begin{lem}\label{lemma:epsilondelta0}
Let $\varepsilon>0$ and $\Lambda \subset \Z^2$ be defined by
\[
\Lambda := \left\{ (l_1,l_2) \in \Z^2 \, : \, -M_i \leq l_i \leq M_i, i=1,2 \right\},
\]
where $M_1,M_2 \in \N$. Furthermore, suppose $A$ is a scaling matrix, and let $\Omega^{(\varepsilon)} := \overline{\Lambda}^{\varepsilon A}$.
If $\varepsilon \longrightarrow 0$, then $\delta^{(\varepsilon)} \longrightarrow 0$, where $\delta^{(\varepsilon)}$ denotes
 the mesh norm of $\{ \varepsilon A l \in \Omega^{(\varepsilon)} \, : \, l \in \Lambda \}$.
\end{lem}

\begin{proof}
If $\varepsilon \rightarrow 0$, then
\[
    \varepsilon A (\Lambda) = \{ \varepsilon A l \, : \, l \in \Lambda \} \longrightarrow \{ (0,0) \}
\]
with respect to the standard Euclidean distance. Furthermore, for $x_l := \varepsilon A l$, we have
\beq
\label{eq:lemma_mesh_proof1}
\lim \limits_{\varepsilon \rightarrow 0} \delta^{(\varepsilon)}
=  \lim \limits_{\varepsilon \rightarrow 0}  \max \limits_{x \in \Omega^{(\varepsilon)}} \min \limits_{l \in \Lambda} \min \limits_{k \in \varepsilon A(\Z^2)} \| x - x_l + k\|_\infty
\leq  \lim \limits_{\varepsilon \rightarrow 0}  \max \limits_{x \in \Omega^{(\varepsilon)}} \min \limits_{l \in \Lambda}  \| x - x_l \|_\infty.
\eeq
Since $x, x_l \in \Omega^{(\varepsilon)}$ for all $ l \in \Lambda$, we obtain
\[
    \lim \limits_{\varepsilon \rightarrow 0} 2 \max \limits_{x \in \Omega^{(\varepsilon)}} \min \limits_{l \in \Lambda}  \| x - x_l \|_\infty
    \leq \lim \limits_{\varepsilon \rightarrow 0}  \max \limits_{x,y \in \Omega^{(\varepsilon)}}  \| x - y \|_\infty.
\]
Inserting this estimate into \eqref{eq:lemma_mesh_proof1} yields
\[
\lim \limits_{\varepsilon \rightarrow 0} \delta^{(\varepsilon)}
\leq \lim \limits_{\varepsilon \rightarrow 0}  \max \limits_{x,y \in \Omega^{(\varepsilon)}}  \| x - y \|_\infty
    \leq \lim \limits_{\varepsilon \rightarrow 0} \diam \Omega^{(\varepsilon)}
    = \lim \limits_{\varepsilon \rightarrow 0} \diam \overline{\Lambda}^{\varepsilon A}
    = \lim \limits_{\varepsilon \rightarrow 0} \diam \varepsilon \overline{\Lambda}^A
    = \lim \limits_{\varepsilon \rightarrow 0} \varepsilon \underbrace{ \diam \overline{\Lambda}^A}_{<\infty} =0,
\]
where $\diam (F)$ denotes the diameter of a set $F\subset \R^d$, i.e.,
\[
    \diam (F) = \sup \limits_{x,y \in F} d_2(x,y),
\]
and $d_2$ denotes the Euclidean metric on $\R^d$. Since the mesh norm is always non-negative, the lemma is proven.
\end{proof}


\subsubsection{Voronoi Regions}

The next result studies the volume of the Voronoi regions associated to the previously considered grid $\Lambda$
with respect to the metric $\rho$ defined in Definition \ref{defi:meshes}. We start by formally defining the
notion of Voronoi region in our setting.

\begin{deff}
Let $\Omega \subset \R^2$, and let $(x_l)_{l \in \Lambda}$ be a sequence in $\R^2$. Then we refer to the
sets $(V_l)_{l \in \Lambda}$ defined by
\[
    V_l := \{ x \in \Omega \, : \, \rho(x,x_l) \leq \rho(x,x_k) \; \text{ for all } k \neq l \}
\]
as \emph{Voronoi regions with respect to $\Lambda$, $\Omega$, $\rho$, and $x_l$}.
\end{deff}

We can now state the previously announced result.

\begin{lem}\label{lemma:V_llessthanone}
Let $M_1,M_2 \in \N$ and
\[
\Lambda := \left\{ (l_1,l_2) \in \Z^2 \, : \, -M_i \leq l_i \leq M_i, i=1,2 \right\}.
\]
Moreover, let $\Omega =\overline{\Lambda}^{\Id}$, where $\Id$ denotes the $2\times2$-identity matrix, and let
$(V_l)_{l \in \Lambda}$ be the Voronoi regions with respect to $\Lambda$, $\Omega$, $\rho$, and $l$. Then, for all $l \in \Lambda$,
\[
\mu (V_l) \leq 1,
\]
where $\mu$ denotes the 2D Lebesgue measure.
\end{lem}

\begin{proof}
Notice that the Voronoi regions $(V_l)_{l \in \Lambda}$ are in fact rectangles, since the grid is an integer grid
with a constant step-size. Hence, for each $l \in \Lambda$,
\beq \label{eq:lemma_Voronoi_proof1}
    \mu (V_l) = a_{l_1,l_2} \cdot b_{l_1,l_2}, \quad a_{l_1,l_2},b_{l_1,l_2} \in \R,
\eeq
where $a_{l_1,l_2}$ denotes the width and $b_{l_1,l_2}$ the height of the rectangle $V_l$.

Towards a contradiction, assume that $V_l$ does contain two different nodes $x_k$ and $x_l$ with $k \neq l$.
This implies
\begin{align*}
     0\neq \rho(x_k,x_l) \leq \rho(x_l,x_l) =0,
\end{align*}
which is a contradiction. Thus, we can conclude that
\[
a_{l_1,l_2} \le \rho(x_{l_1+1,l_2},x_{l_1,l_2})
\quad \mbox{and} \quad
b_l \le \rho(x_{l_1,l_2},x_{l_1,l_2+1}),
\]
which, by \eqref{eq:lemma_Voronoi_proof1}, proves the claim.
\end{proof}

We next obtain a slight generalization of the previous result.

\begin{lem}\label{lemma:V_llessthandet}
Let $A$ be a (linear) bijective transformation acting on $\R^2$ with matrix representation
\[
    A= \begin{pmatrix} \lambda_1& \lambda_2 \\ \lambda_3 & \lambda_4 \end{pmatrix},
\]
and let $\Lambda$ and $(V_l)_{l \in \Lambda}$ be defined as in Lemma \ref{lemma:V_llessthanone}. Then
\[
    \mu (A(V_l)) \leq |\det A |.
\]
\end{lem}

\begin{proof}
The result follows from Lemma \ref{lemma:V_llessthanone} by an integration by substitution.
\end{proof}


\subsubsection{Trigonometric Polynomials}

The next theorem is an adapted version of a result presented in \cite[Thm. 3.2]{PT} which again is a reformulation of a result proven by Gr\"ochenig in \cite{Gro92}.

\begin{prop}\label{pottstheorem}
Let $J, L_1,L_2, L_3, L_4 \in \Z$ such that $L_1 \geq L_3$ and $L_2 \geq L_4$, and let $\Phi$ a trigonometric polynomial
of the form
\[
    \Phi(z) =  \sum \limits_{m_1 = L_3}^{L_1} \sum \limits_{m_2 = L_4}^{L_2} \alpha_{m_1,m_2} e^{-2\pi i  \langle z, m \rangle}, \quad z \in \R^2, m = (m_1,m_2).
\]
Further, let the grid $\Lambda$ be defined as in Lemma \ref{lemma:V_llessthanone}, let $A$ be a scaling matrix, and let
$\Omega^{(\varepsilon)}  := \overline{\Lambda}^{\varepsilon A^{-J}}$ for $\varepsilon >0$. Set $x_l :=\varepsilon(A^{-J})^Tl$,
$l \in \Lambda$. If the mesh norm $\delta$ of $\{ x_l \in \Omega^{(\varepsilon)} \, : \, l \in \Lambda\}$ obeys
\begin{align*}
    \delta  < \frac{\log\left(\frac{1}{\sqrt{\mu ( \Omega^{(\varepsilon)})}}+1\right)}{2 \pi \max\{|L_1|,|L_2|,|L_3|,|L_4|\}},
\end{align*}
where $\mu$ is the 2D Lebesgue measure, then there exists a positive constant $C(\delta,L_1,L_2,L_3,L_4)$ such that
\[
    C(\delta,L_1,L_2,L_3,L_4) \| \Phi \|
    \leq \left( \sum \limits_{l \in \Lambda} \varepsilon^2 |\det A^{-J} | | \Phi(\varepsilon(A^{-J})^Tl|^2 \right)^{\frac{1}{2}},
\]
where
\[
    C(\delta, L_1,L_2,L_3,L_4) = \left( 1 - \left(e^{2\pi \delta \max\{|L_1|,|L_2|,|L_3|,|L_4|\}} -1\right) \sqrt{ \mu(\Omega^{(\varepsilon)})}\right)
\]
and
\begin{align*}
\| f \| := \left(\int_{\Omega^{(\varepsilon)}} |f(x)|^2 \, dx\right)^{1/2}, \quad f\in  L^2(\Omega^{(\varepsilon)}).
\end{align*}

\end{prop}

\begin{proof}
We first observe that by the hypotheses, the constant $C(\delta, L_1,L_2,L_3,L_4)$ is indeed positive.

Second, let $(V_l)_{l \in \Lambda}$ be the Voronoi regions with respect to $\Lambda$, $\Omega^{(\varepsilon)}$, and $\rho$.
For $l \in \Lambda$, we define the weights $\omega_l := \mu(V_l)$. As in Lemma \ref{lemma:V_llessthandet}, integration by substitution
yields
\[
    \omega_l = \mu(V_l) \leq \varepsilon^2 |\det A^{-J} |.
\]
Hence it suffices to prove the existence of a constant $C(\delta, L_1,L_2,L_3,L_4) >0$ such that
\begin{align}
    C(\delta,L_1,L_2,L_3,L_4) \| \Phi \| \leq \left( \sum \limits_{l \in \Lambda} \omega_l | \Phi(\varepsilon(A^{-J})^Tl|^2 \right)^{\frac{1}{2}},  \label{inequality:toprove}
\end{align}

For this, we first observe that
\beq \label{eq:prop_trig_proof1}
\left( \sum \limits_{l \in \Lambda} \omega_l | \Phi(\varepsilon(A^{-J})^Tl|^2 \right)^{\frac{1}{2}}
 = \left( \int \limits_{\Omega^{(\varepsilon)}} \sum \limits_{l \in \Lambda} | \Phi(\varepsilon(A^{-J})^Tl|^2 \right)^{\frac{1}{2}}
 = \left \| \sum \limits_{l \in \Lambda} \Phi(x_l)\chi_{V_l}  \right\|.
\eeq
By the (inverse) triangle inequality,
\begin{align}
    \left \| \sum \limits_{l \in \Lambda} \Phi(x_l)\chi_{V_l}  \right\|
    \geq \| \Phi \| - \left \| \Phi - \sum \limits_{l \in \Lambda}\Phi(x_l)\chi_{V_l}  \right\|. \label{inequality:phiinequality}
\end{align}
Hence, we require an upper bound for $\| \Phi - \sum_{l \in \Lambda}  \Phi(x_l)\chi_{V_l} \|$. By Taylor expansion and Bernstein's inequality,

\begin{align*}
    \left \| \Phi - \sum \limits_{l \in \Lambda} \Phi(x_l) \chi_{V_l}  \right\|^2
    & = \int \limits_{\Omega^{(\varepsilon)}} \left| \Phi(x) - \sum \limits_{l \in \Lambda}  \Phi(x_l)\chi_{V_l}(x) \right|^2 \, dx  \\
    & = \int \limits_{\Omega^{(\varepsilon)}} \left | \sum \limits_{l \in \Lambda}  \Phi(x) \chi_{V_l}(x) - \Phi(x_l)\chi_{V_l}(x) \right|^2 \, dx  \\
    & \leq  \sum \limits_{l \in \Lambda} \int \limits_{V_l} \left|   \Phi(x) - \Phi(x_l) \right|^2 \, dx  \\
    & \leq  \sum \limits_{l \in \Lambda} \int \limits_{V_l} \left|   \sum_{\alpha \in \N^2\setminus\{(0,0)\}}\frac{1}{\alpha!} \| x - x_l\|^\alpha | D^\alpha  \Phi(x)| \right|^2 \, dx  \\
    & \leq  \sum \limits_{l \in \Lambda} \int \limits_{V_l} \left|   \sum_{\alpha \in \N^2\setminus\{(0,0)\}}\frac{1}{\alpha!} \delta^{|\alpha|} ( \max\{|L_1|,|L_2|,|L_3|,|L_4|\}\pi)^{|\alpha|} \|\Phi\| \right|^2 \, dx  \\
    & \leq  \left|   \sum_{\alpha \in \N^2\setminus\{(0,0)\}}\frac{1}{\alpha!} \delta^{|\alpha|} ( \max\{|L_1|,|L_2|,|L_3|,|L_4|\} \pi)^{|\alpha|} \|\Phi\| \right|^2
    \cdot \sum \limits_{l \in \Lambda} \int \limits_{V_l} 1\, dx
\end{align*}
Since the Voronoi regions build a partition of $\Omega^{(\varepsilon)}$,
\[
\sum \limits_{l \in \Lambda} \mu(V_l) = \mu(\Omega^{(\varepsilon)}),
\]
and we can continue by
\beq
    \left \| \Phi - \sum \limits_{l \in \Lambda} \Phi(x_l) \chi_{V_l}  \right\|^2 \leq \left | \left(e^{2\pi \delta \max\{|L_1|,|L_2|,|L_3|,|L_4|\}} - 1\right) \right|^2
    \|\Phi\|^2 \mu(\Omega^{(\varepsilon)}).  \label{proof:estimate}
\eeq
Using (\ref{proof:estimate}) and (\ref{inequality:phiinequality}), we obtain
\begin{align*}
    \left \| \sum \limits_{l \in \Lambda} \Phi(x_l)\chi_{V_l}  \right\| &\geq \| \Phi \| -  \left\| \Phi - \sum \limits_{l \in \Lambda}\Phi(x_l)\chi_{V_l}  \right\| \\
    &\geq \| \Phi \| -  \left| \|\Phi\|  \left(e^{2\pi \delta \max\{|L_1|,|L_2|,|L_3|,|L_4|\}} - 1\right) \sqrt{ \mu(\Omega^{(\varepsilon)})}\right|\\
    &= \| \Phi \| \left( 1 - \left(e^{2\pi \delta \max\{|L_1|,|L_2|,|L_3|,|L_4|\}} -1\right) \sqrt{ \mu(\Omega^{(\varepsilon)})}\right).
\end{align*}
Combining this estimate with \eqref{eq:prop_trig_proof1} proves \eqref{inequality:toprove}.
\end{proof}

Finally, we can state and prove Theorem \ref{maintheoremdelta}, which is one main ingredient for the proof of Theorem \ref{maintheoremgeneral}
in Subsection \ref{subsec:wavelets_structure_proof}.

\begin{theorem}\label{maintheoremdelta}
Let $L_1,L_2, L_3, L_4 \in \Z$ such that $L_1 \geq L_3$ and $L_2 \geq L_4$ and $\alpha_{m_1,m_2} \in \C$, and let $\Phi$ be the trigonometric polynomial
\[
    \Phi ( \cdot, \cdot) = \sum_{m_1= L_3}^{L_1} \sum_{m_2= L_4}^{L_2} \alpha_{m_1,m_2} e^{-2 \pi i (\cdot)m_1}e^{-2 \pi i (\cdot)m_2}.
\]
Further, let $A=\begin{pmatrix} \lambda_1 & \lambda_2 \\ \lambda_3 & \lambda_4 \end{pmatrix}$ be a scaling matrix, $J \in \N$ a maximal
scale, and $\Lambda \subset \Z^2$ the grid defined by
\[
\Lambda := \left\{ (l_1,l_2) \in \Z^2 \, : \, -M_i \leq l_i \leq M_i, i=1,2 \right\},
\]
where $M_1,M_2 \in \N$. If there exists $\varepsilon \leq 1/(T_1 + T_2)$ independent of $J$ such that $\delta$ fulfills
\begin{align*}
    \delta  < \frac{\log\left(\frac{1}{\sqrt{\mu ( \Omega^{(\varepsilon)})}}+1\right)}{ 4 \pi \max\{|L_1|,|L_2|,|L_3|,|L_4|\}},
\end{align*}
 then there exists $\widetilde{S}^{(\theta)} =(S^{(\theta)}_1,S^{(\theta)}_2) \in \N \times \N$,
independent of $J$, such that, for all $\theta >1$, we have
\[
    \sum \limits_{l \in \Lambda} \varepsilon ^2 | \det A^{-J}| | \Phi(\varepsilon (A^{-J})^T l) |^2
    | \phihat(\varepsilon(A^{-J})^Tl)|^2 \geq  \frac{\|\Phi\|^2}{2\theta^2},
\]
for $M= (M_1,M_2) = \frac{1}{\varepsilon} (A^J)^T \widetilde{S}^{(\theta)}$ and scaling function $\phi$.
\end{theorem}

\begin{proof}
Since $0<\varepsilon<1$, there exists an $m\in \N$, such that
\[
\frac{1}{m+1} \le \varepsilon < \frac{1}{m}.
\]
Set $\varepsilon = \frac{1}{m+1}$, and note that \eqref{assumption} still holds, since the logarithm is monotonically increasing.

Next, for some $S \in \N$ to be determined later, let
\[
\begin{pmatrix} M_1 \\ M_2 \end{pmatrix} := \frac{1}{\varepsilon} (A^J)^T\begin{pmatrix} S \\
S \end{pmatrix} = \frac{1}{\varepsilon}\begin{pmatrix} \lambda^{(J)}_1 + \lambda^{(J)}_3 \\ \lambda^{(J)}_2 + \lambda^{(J)}_4 \end{pmatrix} S,
\quad \mbox{where as before }
    A^J=\begin{pmatrix} \lambda^{(J)}_1 & \lambda^{(J)}_2 \\ \lambda^{(J)}_3 & \lambda^{(J)}_4 \end{pmatrix}.
\]
For $l = (l_1,l_2) \in \Z^2$, we then obtain
\begin{eqnarray} \nonumber
\lefteqn{\sum \limits_{l \in \Lambda} \varepsilon ^2  | \det  A^{-J}| | \Phi(\varepsilon (A^{-J})^T l) |^2 | \phihat(\varepsilon(A^{-J})^Tl)|^2}\\ \nonumber
&=&\sum \limits_{l_1 = - M_1}^{M_1 -1} \sum \limits_{ l_2 = M_2 }^{M_2-1}  \varepsilon ^2  | \det  A^{-J}| | \Phi(\varepsilon (A^{-J})^T l) |^2 |
\phihat(\varepsilon(A^{-J})^Tl)|^2\\ \nonumber
&=& \sum \limits_{l_1=0 }^{\frac{\lambda^{(J)}_1 + \lambda^{(J)}_3}{\varepsilon} -1} \sum \limits_{l_2=0}^{\frac{\lambda^{(J)}_2 + \lambda^{(J)}_4}{\varepsilon} -1}
\sum \limits_{s=-S} ^{S-1}\sum \limits_{t=-S} ^{S-1}  \bigg(  \varepsilon ^2 | \det A^{-J}| \cdot
\left| \Phi\left(\varepsilon (A^{-J})^T \left(l+\frac{1}{\varepsilon} (A^J)^T\begin{pmatrix} s \\ t \end{pmatrix}\right) \right) \right|^2 \\
& &  \cdot \left| \phihat\left(\varepsilon (A^{-J})^T \left(l+\frac{1}{\varepsilon} (A^J)^T\begin{pmatrix} s \\ t \end{pmatrix}\right) \right) \right|^2\bigg), \label{proof:maintheoremAJ}
\end{eqnarray}
Integer periodicity of the trigonometric polynomial implies that, for all $s \in \Z$,
\[
    \Phi\left(\varepsilon (A^{-J})^T \left(l+\frac{1}{\varepsilon} (A^J)^T \begin{pmatrix} s \\ t \end{pmatrix} \right) \right)
    = \Phi\left(\varepsilon (A^{-J})^T l + \begin{pmatrix} s \\ t \end{pmatrix}\right)
    = \Phi\left(\varepsilon (A^{-J})^T l \right).
\]
Therefore, by (\ref{proof:maintheoremAJ})
\begin{align}
\sum \limits_{l \in \Lambda} \; &  \varepsilon ^2  | \det  A^{-J}| | \Phi(\varepsilon (A^{-J})^T l) |^2 | \phihat(\varepsilon(A^{-J})^Tl)|^2  \nonumber \\
    &=  \sum \limits_{l_1=0 }^{\frac{\lambda^{(J)}_1 + \lambda^{(J)}_3}{\varepsilon} -1}\sum \limits_{l_2=0}^{\frac{\lambda^{(J)}_2 + \lambda^{(J)}_4}{\varepsilon} -1}  \bigg(\varepsilon ^2 | \det A^{-J}| | \Phi(\varepsilon (A^{-J})^T l) |^2  \sum \limits_{s=-S} ^{S-1} \sum \limits_{t=-S} ^{S-1} \left| \phihat\left((A^{-J})^Tl+\begin{pmatrix} s \\ t \end{pmatrix}\right)\right|^2\bigg). \label{proof:mt2}
\end{align}
Let $\theta >1$. Then, by Lemma \ref{lemma:orthosumone}, there exists $S^{(\theta)} \in \N$ such that
\begin{align}
    \sum \limits_{s=-S^{(\theta)}} ^{S^{(\theta)}-1} \sum \limits_{t=-S} ^{S-1} \left| \phihat\left((A^{-J})^Tl+\begin{pmatrix} s \\ t \end{pmatrix}\right)\right|^2 \geq \frac{1}{\theta} \ . \label{proof:mt3}
\end{align}

We now choose $S := S^{(\theta)}$. Combining (\ref{proof:mt2}) and (\ref{proof:mt3}) yields
\begin{align}
\sum \limits_{l \in \Lambda} \ & \varepsilon ^2  | \det  A^{-J}| | \Phi(\varepsilon (A^{-J})^T l) |^2 | \phihat(\varepsilon(A^{-J})^Tl)|^2 \nonumber \\
    &\geq  \sum \limits_{l_1=0 }^{\frac{\lambda^{(J)}_1 + \lambda^{(J)}_3}{\varepsilon} -1}
    \sum \limits_{l_2=0}^{\frac{\lambda^{(J)}_2 + \lambda^{(J)}_4}{\varepsilon} -1}  \varepsilon ^2 | \det A^{-J}| | \Phi(\varepsilon (A^{-J})^T l) |^2
    \cdot \frac{1}{\theta} \ . \label{proof:mt4}
\end{align}
Next, we apply Proposition \ref{pottstheorem} to (\ref{proof:mt4}) to obtain
\begin{align*}
\sum \limits_{l \in \Lambda} \varepsilon ^2  | \det  A^{-J}|& | \Phi(\varepsilon (A^{-J})^T l) |^2 | \phihat(\varepsilon(A^{-J})^Tl)|^2 \nonumber \\
    &\geq \left(\|\Phi\|\left( 1 - \left(e^{2\pi \delta \max\{|L_1|,|L_2|,|L_3|,|L_4|\}} -1\right) \sqrt{ \mu(\Omega^{(\varepsilon)})}\right) \right)^2
    \cdot \frac{1}{\theta}
\end{align*}
Since
\begin{align*}
1+ &\left( 1- e^{2\pi \delta \max\{|L_1|,|L_2|,|L_3|,|L_4|\}}\right) \sqrt{\mu(\Omega^{(\varepsilon)})} \\
  &\geq 1+ \left( 1- e^{2\pi \max\{|L_1|,|L_2|,|L_3|,|L_4|\} \frac{\log\left(\frac{1}{\sqrt{\mu(\Omega^{(\varepsilon)})}}+1\right)}{4\pi  \max\{|L_1|,|L_2|,|L_3|,|L_4|\}} }\right) \sqrt{\mu(\Omega^{(\varepsilon)})}  \\
&=1+ \left( 1- \left(\sqrt{\mu(\Omega^{(\varepsilon)})}+1\right)^{1/2}\right)\sqrt{\mu(\Omega^{(\varepsilon)})}  \\
&=1+ \left( \sqrt{\mu(\Omega^{(\varepsilon)})}- \left(\mu(\Omega^{(\varepsilon)})^{3/2}+\mu(\Omega^{(\varepsilon)})\right)^{1/2}\right)   \\
&\geq 1 -\mu(\Omega^{(\varepsilon)})^{3/4}\\
&\geq 1- \varepsilon^{3/4} \left( \det A^{-J}\right)^{3/4}\\
&\geq \frac{1}{2}
\end{align*}
we can conclude
\begin{align*}
\sum \limits_{l \in \Lambda} \varepsilon ^2  | \det  A^{-J}|& | \Phi(\varepsilon (A^{-J})^T l) |^2 | \phihat(\varepsilon(A^{-J})^Tl)|^2  \geq \frac{\|\Phi\|^2}{2\theta}.
\end{align*}
Hence the theorem is proven for
\[
M_1 = \frac{ \lambda_1^{(J)} +  \lambda_3^{(J)}}{\varepsilon} S^{(\theta)}
\quad \mbox{and} \quad
M_2 = \frac{ \lambda_2^{(J)} +  \lambda_4^{(J)}}{\varepsilon} S^{(\theta)}.
\]
\end{proof}


\subsubsection{Lemma \ref{lemma:oneepsenough}}

We mention that this result is proved in the same manner as the one dimensional result from \cite{AHP2}.

\begin{lem}\label{lemma:oneepsenough}
For $\gamma >1$ and $\varepsilon_1, \varepsilon_2 \in(0,1/(T_1+T_2)]$, let $\theta(\gamma),C(\gamma)>1$ be such that
\begin{align}
\sqrt{\frac{1}{\theta(\gamma)^{2}} - \frac{16}{\pi^4(C(\gamma) -1)^2}} - \sqrt{1-\frac{1}{\theta(\gamma)}} > \frac{1}{\gamma}. \label{eq:epsclaim}
\end{align}
If there exists $M=(M_1,M_2) \in \N\times\N$ such that
\beq
\label{eq:hypothesis_oneepsenough}
    \inf \limits_{ \substack{\varphi\in\Rcal_N\\ \| \varphi\| =1}} \| P_{\Scal_M^{(\varepsilon_1)}} \varphi \| \geq \frac{1}{\theta(\gamma)},
\eeq
then
\[
    \inf \limits_{ \substack{\varphi\in\Rcal_N\\ \| \varphi\| =1}} \| P_{\Scal_K^{(\varepsilon_2)}} \varphi \| \geq \frac{1}{\gamma},
\]
whenever $K = (K_1,K_2) =\left( \left\lceil \frac{C(\gamma)M_1\varepsilon_1}{\varepsilon_2}\right\rceil,
\left\lceil \frac{C(\gamma)M_2\varepsilon_1}{\varepsilon_2}\right\rceil\right)$.
\end{lem}

\begin{proof}
First, notice that for any $\gamma \in (0,1)$, there exist $\theta(\gamma)$ and $C(\gamma)$ such that (\ref{eq:epsclaim}) is fulfilled.
Now, let $\gamma >1$ and $\varepsilon_2 >0$. Then, by \eqref{eq:hypothesis_oneepsenough},
\begin{align}\nonumber
    \inf \limits_{ \substack{\varphi\in\Rcal_N\\ \| \varphi\| =1}} \| P_{\Scal_K^{(\varepsilon_2)}} \varphi \|
    & \geq \inf \limits_{ \substack{\varphi\in\Rcal_N\\ \| \varphi\| =1}} \left(  \| P_{\Scal_K^{(\varepsilon_2)}} P_{\Scal_M^{(\varepsilon_1)}} \varphi \|
    - \| P_{\Scal_K^{(\varepsilon_2)}} P_{\Scal_M^{(\varepsilon_1)}}^{\perp} \varphi \| \right) \\ \label{eq:proof_oneepsenough_1}
    & \geq \inf \limits_{ \substack{\varphi\in\Rcal_N\\ \| \varphi\| =1}} \left(  \| P_{\Scal_K^{(\varepsilon_2)}} P_{\Scal_M^{(\varepsilon_1)}} \varphi \|
    - \sqrt{1 - \frac{1}{\theta(\gamma)}} \right).
\end{align}
In order to estimate $\| P_{\Scal_K^{(\varepsilon_2)}} P_{\Scal_M^{(\varepsilon_1)}} \varphi \|$, we decompose this term by
\beq
\label{eq:proof_oneepsenough_2}
 \| P_{\Scal_K^{(\varepsilon_2)}} P_{\Scal_M^{(\varepsilon_1)}} \varphi \|^2
 = \| P_{\Scal_M^{(\varepsilon_1)}} \varphi \|^2 - \| P_{\Scal_K^{(\varepsilon_2)}}^{\perp} P_{\Scal_M^{(\varepsilon_1)}} \varphi \|^2.
\eeq
Thus, we require a suitable upper bound for $\| P_{\Scal_K^{(\varepsilon_2)}}^{\perp} P_{\Scal_M^{(\varepsilon_1)}} \varphi \|^2$. For this,
let
\begin{align*}
    I_M &= \{l =(l_1, l_2) \in \Z^2 \, : \, -M_i/2 \leq l_i \leq M_i/2-1, i =1,2 \}, \\
    I_K &= \{j = (j_1, j_2) \in \Z^2 \, : \, -K_i/2 \leq j_i \leq K_i/2 -1, i =1,2 \},
\end{align*}
and for the complementary sets in $\Z^2$ we write
\begin{align*}
    I_M^c &= \{l=(l_1, l_2) \in \Z^2 \, : l \notin I_M \}, \\
    I_K^c &= \{j = (j_1, j_2) \in \Z^2 \, : \, j \notin I_K \}.
\end{align*}
Then, we have
\begin{align}\nonumber
    \| P_{\Scal_K^{(\varepsilon_2)}}^{\perp} P_{\Scal_M^{(\varepsilon_1)}} \varphi \|^2
    &= \left \| \sum_{j \in I_K^c} \langle P_{\Scal_M^{(\varepsilon_1)}} \varphi, s_j^{(\varepsilon_2)} \rangle s_j^{(\varepsilon_2)} \right \|^2 \\ \nonumber
    &=  \sum_{j \in I_K^c}  \left |  \sum_{l \in I_M}  \langle \varphi, s_l^{(\varepsilon_1)} \rangle
    \langle s_l^{(\varepsilon_1)}, s_j^{(\varepsilon_2)} \rangle \right |^2 \\ \nonumber
    &\leq \sum_{j\in I_K^c}  \left(   \sum_{l \in I_M}  | \langle \varphi, s_l^{(\varepsilon_1)} \rangle|^2
    \sum_{l \in I_M}  |\langle s_l^{(\varepsilon_1)}, s_j^{(\varepsilon_2)} \rangle |^2 \right)\\ \label{eq:proof_oneepsenough_3}
    &\leq \sum_{j \in I_K^c}  \sum_{l \in I_M}  |\langle s_l^{(\varepsilon_1)}, s_j^{(\varepsilon_2)} \rangle |^2 .
\end{align}
Now, for $\varepsilon = \max \{ \varepsilon_1, \varepsilon_2 \}$, we obtain
\begin{align*}
    | \langle s_l^{(\varepsilon_1)}, s_j^{(\varepsilon_2)} \rangle |
    &= \left| \varepsilon_1 \cdot \varepsilon_2 \cdot \int \limits_{\left [-\frac{1}{2\varepsilon}, \frac{1}{2\varepsilon} \right]^2}
    e^{2 \pi i \varepsilon_1 \langle l, x \rangle} e^{2 \pi i \varepsilon_2 \langle j, x \rangle} \, dx \right| \\
    &\leq \left | \frac{\varepsilon_1 \varepsilon_2}{\pi^2 (\varepsilon_1 l_1 - \varepsilon_2 j_1)( \varepsilon_1 l_2 -  \varepsilon_2 j_2)} \right|.
\end{align*}
Therefore, by using \eqref{eq:proof_oneepsenough_3},
\[
    \| P_{\Scal_K^{(\varepsilon_2)}}^{\perp} P_{\Scal_M^{(\varepsilon_1)}} \varphi \|^2 \leq \sum_{j \in I_K^c}  \sum_{l \in I_M} \left | \frac{\varepsilon_1 \varepsilon_2}{\pi^2 (\varepsilon_1 l_1 - \varepsilon_2 j_1)( \varepsilon_1 l_2 -  \varepsilon_2 j_2)} \right|^2 .
\]
Assuming $K_i = \frac{C(\gamma) M_i \varepsilon_1}{\varepsilon_2}, i = 1,2$, we can continue by
\begin{align*}
    \| P_{\Scal_K^{(\varepsilon_2)}}^{\perp} P_{\Scal_M^{(\varepsilon_1)}} \varphi \|^2 & \leq \left( \frac{\varepsilon_1 \varepsilon_2}{\pi^2} \right)^2 M_1 M_2 \sum_{(j_1,j_2) \notin I_K}  \frac{4}{ |( \varepsilon_1 \frac{M_1}{2} -  \varepsilon_2 j_1 )( \varepsilon_1 \frac{M_2}{2} -  \varepsilon_2 j_2)|^2}  \\
    & \leq \left( \frac{\varepsilon_1 \varepsilon_2}{\pi^2} \right)^2 M_1 M_2   \frac{16}{\varepsilon_2^2 |( \varepsilon_1 M_1 -  \varepsilon_2 K_1 )( \varepsilon_1 M_2 -  \varepsilon_2 K_2)|}  \\
    & \leq \left( \frac{\varepsilon_1 }{\pi^2} \right)^2 M_1 M_2   \frac{16}{ |( \varepsilon_1 M_1 -  \varepsilon_2 K_1 )( \varepsilon_1 M_2 -  \varepsilon_2 K_2)|}  \\
    & \leq \left( \frac{\varepsilon_1 }{\pi^2} \right)^2 M_1 M_2   \frac{16}{ |( \varepsilon_1 M_1 -   C(\gamma) M_1 \varepsilon_1 )( \varepsilon_1 M_2 -   C(\gamma) M_2 \varepsilon_1)|}  \\
    &= \frac{16}{\left( \pi^2 (C(\gamma )-1) \right)^2}.
\end{align*}
Thus, by \eqref{eq:proof_oneepsenough_2},
\[
    \| P_{\Scal_K^{(\varepsilon_2)}} P_{\Scal_M^{(\varepsilon_1)}} \varphi \|^2
    \geq \frac{1}{\theta(\gamma)^2} - \frac{16}{\left( \pi^2 (C(\gamma )-1) \right)^2}
\]
which, using \eqref{eq:proof_oneepsenough_1}, yields
\[
    \inf \limits_{ \substack{\varphi\in\Rcal_N\\ \| \varphi\| =1}} \| P_{\Scal_K^{(\varepsilon_2)}} \varphi
    \| \geq \sqrt{\frac{1}{\theta(\gamma)^2} - \frac{16}{\left( \pi^2 (C(\gamma )-1) \right)^2}} - \sqrt{1 - \frac{1}{\theta}}.
\]
The lemma is proved.
\end{proof}


\section{Proof of Theorem \ref{theorem:boundarywavelets}}
\label{sec:proof_b_wavelets}

The following lemma will be used in the upcoming proof Theorem \ref{theorem:boundarywavelets}. A one dimensional analogue can be found in \cite{AHP2}. The proof extends straightforwardly and we omit it here.
\begin{lem} \label{lemma:trig}
Let $A_1, A_2, A_3$, and $A_4 \in \Z, A_1 \leq A_2, A_3 \leq A_4$. Moreover, let $L_1, L_2 \in \N$ such that $2 L_1 \geq A_2 - A_1 +1$ and $2L_2 \geq A_4 - A_3 +1$. Then the trigonometric polynomial $$\Phi(z_1,z_2) =  \sum_{k = A_1}^{A_2} \sum_{l=A_3}^{A_4} \alpha_{k,l} e^{2\pi i k z_1} e^{2\pi i l z_2}$$ satisfies
\begin{align*}
	\sum \limits_{m=0}^{2L_1 - 1} \sum \limits_{n=0}^{2L_2 - 1}  \frac{1}{4L_1 L_2} \left| \Phi\left(\frac{m}{2L_1}, \frac{n}{2L_2}\right) \right|^2 =  \sum_{k = A_1}^{A_2} \sum_{l=A_3}^{A_4} |\alpha_{k,l} |^2 .
\end{align*}
\end{lem}

\begin{proof}[Proof of Theorem \ref{theorem:boundarywavelets}]
Let $\varphi \in \Rcal_N$ such that $\| \varphi\| =1$. Then $\varphi$ can be expanded as
\begin{align}
    \varphi =  \sum_{n_1,n_2=0}^{2^{J_0}-1} \alpha_{J_0,(n_1,n_2)} \phi_{J_0,(n_1,n_2)} + \sum _{k=1}^3 \sum  _{j = J_0}^{J-1} \sum_{n_1,n_2 =0}^{2^j-1} \beta^k_{j,(n_1,n_2)} \psi_{j,(n_1,n_2)}^{\textint,k}  \label{eq1:proof}
\end{align}
We will now use the nestedness of the two dimensional MRA that is generated by the one dimensional wavelet system $\left\{  \{ \phiint_{J_0,m} \}_{m = 0, \ldots, 2^J-1}, \{ \psiint_{j,n} \}_{j \geq J_0,n = 0, \ldots, 2^j-1}\right\}$. In particular, we have
\begin{align*}
    V_j^{\textint,2} \subset V_{j+1}^{\textint,2} , \quad j\geq J_0,
\end{align*}
and
\begin{align*}
V_{j+1}^{\textint,2} = V_{j}^{\textint,2} \oplus W_{j}^{\textint,2}, \quad j\geq J_0,
\end{align*}
with $V_j^{\textint,2} = V_j^{\textint} \otimes V_j^{\textint}$ and $W_j^{\textint,2} = (V^{\textint}_j \otimes W^{\textint}_j) \oplus (W^{\textint}_j \otimes V^{\textint}_j) \oplus (W^{\textint}_j \otimes W^{\textint}_j)$.
Loosely speaking, due to the MRA embedding properties we can expand functions from the reconstructions space into scaling functions $(\phiint_{J,(n_1,n_2)})_{n_1,n_2}$ at highest scale. Since the left boundary functions can be constructed by translates of the initial scaling function $\phi$ and the right scaling function can be obtained by reflecting the left boundary functions. The reflected function will be denoted by $\phi^{\#}$. This gives us in (\ref{eq1:proof})
\begin{align}
	\varphi = \sum_{n_1,n_2=0}^{2^J-p-1} \alpha_{n_1,n_2} \phi(2^{J} \cdot -n)  + \sum_{n_1,n_2 =2^J-p}^{2^J} \beta_{n_1,n_2} \phi^{\#}(2^{J} \cdot -n),\label{eq2:proof}
\end{align}
where only finitely many $\alpha_{n_1,n_2}$ and $\beta_{n_1,n_2}$ are non-zero. Now, for any $l= (l_1,l_2)  \in \Z^2$ we obtain by basic properties of the Fourier transform
\begin{align}
\langle \varphi&, s_l^{(\varepsilon)} \rangle =   \sum_{n_1,n_2=0}^{2^J-p-1} \alpha_{n_1,n_2} \frac{\varepsilon}{ 2^{J}} e^{-2\pi i \varepsilon \langle n, 2^{-J}l\rangle}\phihat\left(\frac{\varepsilon}{ 2^{J}}l\right)  + \sum_{n_1,n_2 =2^J-p}^{2^J} \beta_{n_1,n_2} \frac{\varepsilon}{ 2^{J}} e^{-2\pi i \varepsilon \langle n, 2^{-J}l\rangle} \widehat{\phi^{\#}}\left(\frac{\varepsilon}{ 2^{J}}l\right). \label{eq3:proof}
\end{align}
For the sake of brevity, we shall write in the following
\begin{align}
	\Phi_1(z) &=  \sum_{n_1,n_2=0}^{2^J-p-1} \alpha_{n_1,n_2} \frac{\varepsilon}{ 2^{J}} e^{-2\pi i \varepsilon \langle n, z\rangle},\nonumber \\
	\Phi_2(z) &= \sum_{n_1,n_2 =2^J-p}^{2^J} \beta_{n_1,n_2} \frac{\varepsilon}{ 2^{J}} e^{-2\pi i \varepsilon \langle n, z\rangle}. \label{eq4:proof}
\end{align}

By our  assumptions on the scaling function
\begin{align}
    |\widehat{\phi}(\xi_1, \xi_2)| \lesssim \frac{1}{ (1+|\xi_1|)(1+ |\xi_2|)}, \label{eq5:proof}
\end{align}
and by the same argument
\begin{align}
    |\widehat{\phi^{\#}}(\xi_1, \xi_2)| \lesssim \frac{1}{(1+|\xi_1|)(1+ |\xi_2|)}. \label{eq6:proof}
\end{align}
Using (\ref{eq4:proof}), (\ref{eq5:proof}), and  (\ref{eq6:proof}) in (\ref{eq3:proof}) yields
\begin{align}
\sum _{l \notin I_M} |\langle \varphi , s_l^{(\varepsilon)} \rangle |^2 \nonumber &\leq \sum _{l \notin I_M} \left| \frac{\varepsilon}{ 2^{J}}\Phi_1\left(\frac{\varepsilon}{2^{J}}l\right) \phihat\left(\frac{\varepsilon}{ 2^{J}}l\right) \right|^2  +  \sum _{l \notin I_M}\left|\frac{\varepsilon}{ 2^{J}} \Phi_2\left(\frac{\varepsilon}{2^{J}}l\right) \widehat{\phi^{\#}}\left(\frac{\varepsilon}{ 2^{J}}l\right) \right|^2 +  \nonumber \\
&\makebox[0.5cm][c]{}+ 2 \left( \sum _{l \notin I_M} \left|\frac{\varepsilon}{ 2^{J}} \Phi_1\left(\frac{\varepsilon}{2^{J}}l\right) \phihat\left(\frac{\varepsilon}{ 2^{J}}l\right) \right|^2 \right)^{1/2} \left( \sum _{l \notin I_M} \left|\frac{\varepsilon}{ 2^{J}} \Phi_2\left(\frac{\varepsilon}{2^{J}}l\right) \widehat{\phi^{\#}}\left(\frac{\varepsilon}{ 2^{J}}l\right) \right|^2\right)^{1/2}. \label{eq7:proof}
\end{align}
We assume $2^J/\varepsilon \in \N$ and the number of samples $M=(M_1,M_2) \in \N \times \N$ is
\begin{align*}
    M_i = S \cdot \frac{2^J}{\varepsilon}, \quad i = 1,2
\end{align*}
where $S$ is some positive constant. Now, $(l_1,l_2) \notin I_M$ if

\textbf{Case \textrm{I}:} $|l_1| > M_1$ and $|l_2|<M_2$,

\textbf{Case \textrm{II}:} $|l_1| < M_1$ and $|l_2|>M_2$, or

\textbf{Case \textrm{III}:} $|l_1| > M_1$ and $|l_2|>M_2$.

It is sufficient to consider the sum for Case \textrm{I}. Case \textrm{II} can be obtained by symmetry and Case \textrm{III} yields a smaller sum. For $K = 2^J/\varepsilon$, we have
\begin{align}
	\sum _{|l_2| <M_2}& \sum _{|l_1| >M_1} \left| \frac{\varepsilon}{ 2^{J}} \Phi_1\left(\frac{\varepsilon}{2^{J}}l\right) \phihat\left(\frac{\varepsilon}{ 2^{J}}l\right) \right|^2 \nonumber \\
	&= \sum_{j_2}^{K-1} \sum_{j_1}^{K-1} \frac{1}{K}\frac{1}{K}\left| \Phi_1\left(\frac{j_1}{K}, \frac{j_2}{K}\right) \right|^2 \sum_{|s_2|<S} \sum_{|s_1|>S} \left|\phihat\left(\frac{j_1}{K} + s_1,\frac{j_2}{K} + s_2\right) \right|^2 \nonumber \\
	&\lesssim \sum_{j_2}^{K-1} \sum_{j_1}^{K-1} \frac{1}{K}\frac{1}{K}\left| \Phi_1\left(\frac{j_1}{K}, \frac{j_2}{K}\right) \right|^2 \sum_{|s_2|<S} \sum_{|s_1|>S} \frac{1}{(1 + |j_1/K + s_1|)^2}\frac{1}{(1 + |j_2/K + s_2|)^2} \nonumber \\
	&\leq C_1 \sum_{j_2}^{K-1} \sum_{j_1}^{K-1} \frac{1}{K}\frac{1}{K}\left| \Phi_1\left(\frac{j_1}{K}, \frac{j_2}{K}\right) \right|^2 \frac{1}{S}. \label{eq8:proof}
\end{align}
By Lemma \ref{lemma:trig} we obtain in (\ref{eq8:proof})
\begin{align}
	\sum _{|l_2| >M_2}& \sum _{|l_1| >M_1} \left| \frac{\varepsilon}{ 2^{J}} \Phi_1\left(\frac{\varepsilon}{2^{J}}l\right) \phihat\left(\frac{\varepsilon}{ 2^{J}}l\right) \right|^2 \leq C_1 \sum_{n_1,n_2=0}^{2^J-1} |\alpha_{n_1,n_2}|^2 \frac{1}{S}. \label{eq8:proof}
\end{align}
Since the functions form an orthonormal basis and $\| \varphi \| =1$ we have
\begin{align*}
	\sum_{n_1,n_2\in \N} |\alpha_{n_1,n_2}|^2 + \sum_{n_1,n_2\in \N} |\beta_{n_1,n_2}|^2 = 1
\end{align*}
and hence
\begin{align*}
	\sum_{n_1,n_2\in \N} |\alpha_{n_1,n_2}|^2 \leq 1
\end{align*}
Similarly, one shows
\begin{align}
	\sum _{|l_2| >M_2}& \sum _{|l_1| >M_1} \left| \frac{\varepsilon}{ 2^{J}} \Phi_2\left(\frac{\varepsilon}{2^{J}}l\right) \widehat{\phi^{\#}} \left(\frac{\varepsilon}{ 2^{J}}l\right) \right|^2 \leq C_2 \sum_{n_1,n_2 \in \N}|\beta_{n_1,n_2}|^2 \frac{1}{S}. \label{eq9:proof}
\end{align}
Therefore, in (\ref{eq7:proof}) we have that
\begin{align*}
\sum _{|l_2| <M_2} \sum _{|l_1| >M_1} |\langle \varphi , s_l^{(\varepsilon)} \rangle |^2 & \leq C_1 \sum_{n_1,n_2\in \N} |\alpha_{n_1,n_2}|^2 \frac{1}{S} + C_2 \sum_{n_1,n_2 \in \N} |\beta_{n_1,n_2}|^2 \frac{1}{S} +\\
&\qquad +  2 \left(C_1 \sum_{n_1,n_2\in \N} |\alpha_{n_1,n_2}|^2 \frac{1}{S}\right)^{1/2} \left(C_2\sum_{n_1,n_2\in \N}|\beta_{n_1,n_2}|^2 \frac{1}{S}\right)^{1/2} \nonumber \\
&\leq \left( \frac{\sqrt{2C_1}+\sqrt{2C_2}}{\sqrt{S}}\right)^2
\end{align*}
Now, for $\theta>1$ choosing $S$ large enough, such that
\begin{align*}
\left( \frac{\sqrt{2C_1}+\sqrt{2C_2}}{\sqrt{S}}\right)^2\leq \frac{\theta^2 - 1}{3\theta^2}
\end{align*}
gives the claim.
\end{proof}

\section{Acknowledgements}

The authors would like to thank Bogdan Roman for providing Figure 2 and Figure 3. BA acknowledges support from the NSF DMS grant 1318894. ACH acknowledges support from a Royal Society University Research Fellowship as well as the UK Engineering and Physical Sciences Research Council (EPSRC) grant EP/L003457/1. GK was supported in part by the Einstein Foundation Berlin, by Deutsche Forschungsgemeinschaft (DFG) Grant KU 1446/14, by the DFG Collaborative Research Center TRR 109 ``Discretization in Geometry and Dynamics'', and by the DFG Research Center {\sc Matheon} ``Mathematics for key technologies'' in Berlin. JM acknowledges support from the Berlin Mathematical School.

\bibliographystyle{abbrv}
\bibliography{paperref}

\begin{thebibliography}{10}

\bibitem{BAMGACHNonuniform1D}
B.~Adcock, M.~Gataric, and A.~C. Hansen.
\newblock On stable reconstructions from univariate nonuniform {F}ourier
  measurements.
\newblock {\em arXiv:1310.7820}, 2013.

\bibitem{AH2}
B.~Adcock and A.~C. Hansen.
\newblock Generalized sampling and infinite-dimensional compressed sensing.
\newblock Technical report, 2011.

\bibitem{AH1}
B.~Adcock and A.~C. Hansen.
\newblock A generalized sampling theorem for stable reconstructions in
  arbitrary bases.
\newblock {\em J. Fourier Anal. Appl.}, 18(4):685--716, 2012.

\bibitem{AH3}
B.~Adcock and A.~C. Hansen.
\newblock Stable reconstructions in {H}ilbert spaces and the resolution of the
  {G}ibbs phenomenon.
\newblock {\em Appl. Comput. Harmon. Anal.}, 32(3):357--388, 2012.

\bibitem{AHHT1}
B.~Adcock, A.~C. Hansen, E.~Herrholz, and G.~Teschke.
\newblock Generalized sampling: extension to frames and inverse and ill-posed
  problems.
\newblock {\em Inverse Problems}, 29(1):015008, 27, 2013.

\bibitem{AHP2}
B.~Adcock, A.~C. Hansen, and C.~Poon.
\newblock Beyond consistent reconstructions: optimality and sharp bounds for
  generalized sampling, and application to the uniform resampling problem.
\newblock {\em SIAM J. Math. Anal.}, 45(5):3132--3167, 2013.

\bibitem{AHP1}
B.~Adcock, A.~C. Hansen, and C.~Poon.
\newblock On optimal wavelet reconstructions from fourier samples: Linearity
  and universality of the stable sampling rate.
\newblock {\em Appl. Comput. Harmon. Anal.}, to appear.

\bibitem{Ald1}
A.~Aldroubi.
\newblock Oblique projections in atomic spaces.
\newblock {\em Proc. Amer. Math. Soc.}, 124(7):2051--2060, 1996.

\bibitem{BHKL1}
S.~Bishop, C.~Heil, Y.~Y. Koo, and J.~K. Lim.
\newblock Invariances of frame sequences under perturbations.
\newblock {\em Linear Algebra Appl.}, 432(6):1501--1574, 2010.

\bibitem{CandesDonoho}
E.~J. Cand\'{e}s and D.~L. Donoho.
\newblock New tight frames of curvelets and optimal representations of objects
  with piecewise c? singularities.
\newblock {\em Comm. Pure Appl. Math}, pages 219--266, 2002.

\bibitem{CohDauVial}
A.~Cohen, I.~Daubechies, and P.~Vial.
\newblock Wavelet bases on the interval and fast algorithms.
\newblock {\em Appl. Comput. Harmon. Anal}, (1):54--81, 1993.

\bibitem{CorachMaest}
G.~Corach and A.~Maestripieri.
\newblock Products of orthogonal projections and polar decompositions.
\newblock {\em Linear Algebra Appl.}, 434(6):1594--1609, 2011.

\bibitem{Dau}
I.~Daubechies.
\newblock {\em Ten lectures on wavelets}, volume~61 of {\em CBMS-NSF Regional
  Conference Series in Applied Mathematics}.
\newblock Society for Industrial and Applied Mathematics (SIAM), Philadelphia,
  PA, 1992.

\bibitem{FDeutsch}
F.~Deutsch.
\newblock {\em Best approximation in inner product spaces}.
\newblock CMS Books in Mathematics/Ouvrages de Math\'ematiques de la SMC, 7.
  Springer-Verlag, New York, 2001.

\bibitem{Eldar1}
Y.~C. Eldar.
\newblock Sampling with arbitrary sampling and reconstruction spaces and
  oblique dual frame vectors.
\newblock {\em J. Fourier Anal. Appl.}, 9(1):77--96, 2003.

\bibitem{Eldar2}
Y.~C. Eldar and T.~Werther.
\newblock General framework for consistent sampling in {H}ilbert spaces.
\newblock {\em Int. J. Wavelets Multiresolut. Inf. Process.}, 3(4):497--509,
  2005.

\bibitem{Gro92}
K.~Gr{\"o}chenig.
\newblock {Reconstruction algorithms in irregular sampling}.
\newblock {\em Math. Comp.}, 59(181--1924), 1992.

\bibitem{hrycakIPRM}
T.~Hrycak and K.~Gr\"{o}chenig.
\newblock Pseudospectral {F}ourier reconstruction with the modified inverse
  polynomial reconstruction method.
\newblock {\em J. Comput. Phys.}, 229(3):933--946, 2010.

\bibitem{HealyWeaver}
D.~M.~H. Jr. and J.~B. Weaver.
\newblock Two applications of wavelet transforms in magnetic resonance imaging.
\newblock {\em IEEE Transactions on Information Theory}, 38(2):840--860, 1992.

\bibitem{KutLim}
G.~Kutyniok and W.-Q. Lim.
\newblock Compactly supported shearlets are optimally sparse.
\newblock {\em Journal of Approximation Theory}, 163(11):1564--1589, 2011.

\bibitem{Maass}
P.~Maass.
\newblock Families of orthogonal 2d wavelets.
\newblock {\em SIAM J. Math. Anal.}, pages 1454--1481, 1996.

\bibitem{Mallat}
S.~Mallat.
\newblock {\em A wavelet tour of signal processing}.
\newblock Elsevier/Academic Press, Amsterdam, third edition, 2009.
\newblock The sparse way, With contributions from Gabriel Peyr{\'e}.

\bibitem{Panych}
L.~Panych.
\newblock Theoretical comparison of fourier and wavelet encoding in magnetic
  resonance imaging.
\newblock 15(2):141--53, 02 1996.

\bibitem{PT}
D.~Potts and M.~Tasche.
\newblock Numerical stability of nonequispaced fast {F}ourier transforms.
\newblock {\em J. Comput. Appl. Math.}, 222(2):655--674, 2008.

\bibitem{Strang}
G.~Strang and T.~Nguyen.
\newblock {\em {Wavelets and filter banks}}.
\newblock Wellesley-Cambridge Press, 1997.

\bibitem{Tang1}
W.-S. Tang.
\newblock Oblique projections, biorthogonal {R}iesz bases and multiwavelets in
  {H}ilbert spaces.
\newblock {\em Proc. Amer. Math. Soc.}, 128(2):463--473, 2000.

\bibitem{unser2000sampling}
M.~Unser.
\newblock Sampling--50 years after {S}hannon.
\newblock {\em Proc. IEEE}, 88(4):569--587, 2000.

\bibitem{UnserAld}
M.~Unser and A.~Aldroubi.
\newblock A general sampling theory for nonideal acquisition devices.
\newblock {\em IEEE Transactions on Signal Processing}, 42(11):2915--2925,
  1994.

\bibitem{UnserAldroubiWaveletReview}
M.~Unser and A.~Aldroubi.
\newblock A review of wavelets in biomedical applications.
\newblock {\em Proc. IEEE}, 84(4):626--638, 1996.

\bibitem{UnserAldroubiLaineEditorial}
M.~Unser, A.~Aldroubi, and A.~F. Laine.
\newblock Guest editorial: wavelets in medical imaging.
\newblock {\em IEEE Trans. Med. Imaging}, 22(3):285--288, 2003.

\end{thebibliography}

\end{document}